\documentclass[amscd,amssymb,verbatim,11pt]{amsart}
\numberwithin{equation}{section}
\usepackage{epsfig}
\usepackage{multido}
\usepackage{color}
\usepackage{pstricks}
\usepackage{pst-all}
\usepackage{pst-math}
\usepackage{pst-func}
\usepackage{pst-3dplot}
\usepackage{graphicx}
\usepackage{amsmath}
\usepackage{a4,amssymb}

\newtheorem{thm}{Theorem}[section]
\newtheorem{cor}[thm]{Corollary}
\newtheorem{lem}[thm]{Lemma}

\theoremstyle{definition}

\newtheorem{rem}[thm]{Remark}

\newif\ifShowLabels
\ShowLabelstrue
\newdimen\theight
\def\TeXref#1{%by M. Doob
     \leavevmode\vadjust{\setbox0=\hbox{{\tt
            \quad\quad  {\small  \bf #1}}}%
     \theight=\ht0
     \advance\theight  by  \dp0
     \advance\theight  by  \lineskip
     \kern -\theight \vbox  to
     \theight{\rightline{\rlap{\box0}}%
      \vss}%
      }}%

    {\begin{thm}\label{#1} \ifShowLabels \TeXref{#1} \fi}%
    {\end{thm}}

    {\begin{def}\label{#1} \ifShowLabels \TeXref{#1} \fi}%
    {\end{def}}

    {\begin{lem}\label{#1} \ifShowLabels \TeXref{#1} \fi}%
    {\end{lem}}

    {\begin{cor}\label{#1} \ifShowLabels \TeXref{#1} \fi}%
    {\end{cor}}

\newcommand{\eqRef}[1]%
     {\ifShowLabels \TeXref{#1} \fi
      \begin{equation}\label{#1} }

\ShowLabelsfalse

\newcommand{\vsp}{\vskip 1em}

\newcommand{\NI}{\noindent}
\newcommand{\bea}{\begin{eqnarray}}
\newcommand{\eea}{\end{eqnarray}}
\newcommand{\IR}{\mathbb{R}}
\newcommand{\bas}{\begin{align*}}
\newcommand{\eas}{\end{align*}}
\newcommand{\ba}{\begin{align}}
\newcommand{\ea}{\end{align}}

\newcommand{\be}{\begin{equation}}
\newcommand{\ee}{\end{equation}}
\newcommand{\ben}{\begin{eqnarray*}}
\newcommand{\een}{\end{eqnarray*}}
\newcommand{\lam}{\lambda}
\newcommand{\Om}{\Omega}

\newcommand{\tht}{\theta}
\newcommand{\p}{\partial}
\newcommand{\al}{\alpha}

\newcommand{\ve}{\varepsilon}
\newcommand{\dl}{\delta}
\newcommand{\D}{\Delta}

\newcommand{\G}{\Gamma}

\newcommand{\lamo}{\lambda_{\Om}}

\newcommand{\Df}{\D_p}
\newcommand{\Dy}{\D_{\infty}}

\title[Asymptotics]{Asymptotics of viscosity solutions to some doubly nonlinear parabolic equations}

\author[Bhattacharya and Marazzi]{Tilak Bhattacharya and Leonardo Marazzi}

\begin{document}

\maketitle

\begin{abstract} We study asymptotic decay rates of viscosity solutions to some doubly nonlinear 
parabolic equations including Trudinger's equation. We also prove a Phragm\'en-Lindel\"of type result and show its optimality.
\end{abstract}

\section{\bf Introduction}

In this work, we prove some results for the viscosity solutions to some doubly nonlinear parabolic equations. The main focus of this paper is Trudinger's equation but we will also state some results for a parabolic equation involving the infinity-Laplacian. This is a follow-up of the works in \cite{BM1, BM2}.

To describe our results more precisely, we introduce definitions and notations. We take $n\ge 2$ in this work. 
Letters like $x,\;y,\;z$ etc, denote the spatial variables, $s,\;t$ the time variables and $o$ stands for the origin in $\IR^n$. Let $\overline{A}$ denote the closure of a set $A$. The ball of radius $R>0$ and center $x\in \IR^n$ is denoted by $B_R(x)$. 
Let $\Om\subset \IR^n$ be a bounded domain and $0<T<\infty$. We define
$\Om_T=\Om\times (0,T)$ and its parabolic boundary as $P_T=(\overline{\Om}\times\{0\})\cup (\p \Om \times (0,T)).$ 

For $2\le p<\infty$, define the $p$-Laplacian $\Df$ and the infinity-Laplacian $\Dy$ as 
\eqRef{1.1}  
\Df u=\mbox{div}(|Du|^{p-2}Du)\;\;\;\mbox{and}\;\;\Dy u=\sum_{i,j=1}D_i uD_j u D_{ij}u,
\ee
where $u=u(x)$. We now define the parabolic operators of interest to us. Call
\eqRef{1.2}
\G_p u=\Df u-(p-1) u^{p-2} u_t,\;\;2\le p<\infty,\;\;\;\mbox{and}\;\;\;\G_{\infty} u=\Dy u-3 u^2 u_t,
\ee
where $u=u(x,t)$. The equation $\G_p=0,\;2\le p<\infty,$ is the well-known Trudinger equation \cite{TR}. See also
\cite{BM1,BM2} and the references therein. The operators $\G_p,\;2< p\le \infty$ are doubly nonlinear and degenerate and, in this work, solutions will be understood to be in the viscosity
sense. Note that we use $p=\infty$ as a label. It is not clear to us what the limit of $\G_p$ (and $G_p$, see below) is for $p\rightarrow \infty$. For a detailed discussion about nonlinear parabolic equations, see \cite{ED}.

Suppose that $0<T<\infty$. Let $f\in C(\overline{\Om})$ and $g(x,t)\in C( \p\Om\times[0,T))$. For ease of notation, we define 
\eqRef{1.30}
h(x,t)=\left\{ \begin{array}{ccc} f(x), & \forall x\in \overline{\Om},\\ g(x,t), & \forall (x,t)\in  \p\Om\times[0,T) \end{array}\right.
\ee
We take $h\in C(P_T)$, in the sense that $\lim_{y\rightarrow x}f(y)=g(x,0)$ for each $x\in \p\Om$.
In most of this work, we take 
\eqRef{1.3} 
0<\inf_{P_T} h(x,t)\le \sup_{P_T}h(x,t)<\infty.
\ee

For $2\le p\le \infty$, we consider positive viscosity solutions $u\in C(\Om_T\cup P_T)$ of 
\eqRef{1.4} 
\G_p u=0,\;\;\mbox{in $\Om_T$ and $u=h$ on $P_T$.}
\ee
In \cite{BM1} (see Theorem 5.2), we showed the existence of positive viscosity solutions of (\ref{1.4}) for $p=\infty$. The work \cite{BM2} showed the existence of positive viscosity solutions
for $2\le p<\infty$, see Theorems 1.1 and 1.2 therein. For the case $2\le p\le n$, this result is proven for domains $\Om$ that satisfy a uniform outer ball condition. For $n<p<\infty$, the result is shown for any $\Om$.

We will also have occasion to work with equations related to $\G_p$. As observed in Lemma 2.2 in \cite{BM1} and Lemma 2.1 in \cite {BM2}, 
if $u>0$ solves the doubly nonlinear equation $\G_p u=0,\;2\le p\le \infty,$ (see (\ref{1.2})), then $v=\log u$ solves
$$\D_pv+(p-1)|Dv|^p-(p-1)v_t=0,\;\;2\le p<\infty,\;\;\mbox{and}\;\;\;\Dy v+|Dv|^4-3v_t=0,$$
For convenience of presentation, call
\eqRef{1.5}
G_p w=\D_pw+(p-1)|Dw|^p-(p-1)w_t,\;\;2\le p<\infty,\;\;\mbox{and}\;\;G_\infty w=\Dy w+|Dw|^4-3w_t.
\ee

We now state the main results of this work. Let $\lamo$ be the first eigenvalue of $\D_p$ on $\Om$.

\begin{thm}\label{1.6} Let $2\le p<\infty$ and $\Om\subset \IR^n,\;n\ge 2,$ be a bounded domain. Call
$\Om_\infty= \Om\times(0, \infty)$ and 
$P_\infty$ its parabolic boundary. Suppose that $h\in C(P_\infty)$ is as defined in (\ref{1.30}) with
$h\ge 0$ and $\sup_{P_\infty} h<\infty$.
Let 
$u\in usc(\Om_\infty\cup P_\infty)),\;u\ge 0,$ solve
$$\G_p u\ge 0,\;\;\mbox{in $\Om_\infty$ and $u\le h$ on $P_\infty$}.$$

(i) If $\lim_{t\rightarrow \infty} (\sup_{\p\Om}g(x,t) )=0$ then $\lim_{t\rightarrow \infty} (\sup_{\Om\times [t, \infty)}u)=0.$ 

(ii) Moreover, if $g(x,t)=0,\forall(x,t)\in \p\Om\times[T_0,\infty)$, for some $T_0\ge 0$, then
$$\lim_{t\rightarrow \infty} \frac{\log(\sup_\Om u(x,t) )}{t} \le -\frac{\lam_\Om}{p-1}. $$
\end{thm}
The above result is an analogue of the asymptotic result proven in Theorem 4.4 and Lemma 4.7 in \cite{BM1} for 
$\G_\infty u\ge 0$. We provide an example where the rate $\exp(-\lamo t/(p-1))$ is attained, see Remark 
\ref{3.15}. Note that we do not address existence for $h\ge 0$. We also show

\begin{thm}\label{1.61} Let $2\le p\le \infty$, $\Om\subset \IR^n,\;n\ge 2,$ be a bounded domain, 
$\Om_\infty= \Om\times(0, \infty)$ and 
$P_\infty$ be its parabolic boundary. Suppose that $h\in C(P_\infty)$ is as defined in (\ref{1.30}). 
Assume that $0<\inf_\Om f\le 1\le \sup_\Om f<\infty$ and $g(x,t)=1,\;\forall(x,t)\in \p\Om\times[0,\infty)$.

If $u\in C(\Om_\infty\cup P_\infty)),\;u>0,$ solves
$$\G_p u= 0,\;\;\mbox{in $\Om_\infty$, $u(x,0)=f(x),\;\forall x\in \overline{\Om}$ and $u(x,t)=1,\;\forall(x,t)\in \p\Om\times[0,\infty)$},$$
then for every $x\in \Om$, $\lim_{t\rightarrow \infty} u(x,t)=1.$ 
\end{thm}
From the proof, it follows that (i) $u(x,t)=\exp(O(t^{-s})),\;p=2$ and any $s>0$, (ii) $u(x,t)=\exp(O(t^{-1/(p-2)})),\;2<p<\infty,$ and (iii)
$u(x,t)=\exp(O(t^{-1/2})),\;p=\infty$, as $t\rightarrow \infty$. 
From the works in \cite{AJK, JL} one sees that (i) for $\D_\infty u=u_t$, the asymptotic decay 
is $t^{-1/2}$ and (ii) for $\Df u=u_t$, the rate is $t^{-1/(p-2)}.$ They do appear to agree if we consider $G_p.$
However, at this time, it is not clear if the asymptotic rates in Theorem \ref{1.61} are optimal and also if $u$ tends to a $p$-harmonic function when $g(x,t)=g(x)$, for all $t>0$.

We now state a Phragm\'en-Lindel\"of type result for the unbounded domain $\IR^n\times(0,T)$,
where $0<T<\infty$. A version was shown in Theorem 4.1 in \cite{BM1} for $\G_\infty$. We show an analogue for $\G_p$, $2\le p < \infty$, and include an improvement for $\G_\infty$.

\begin{thm}\label{1.7} Let $2\le p\le \infty$ and $0<T<\infty$. 
Assume that $0<\inf_{\IR^n}f(x)\le \sup_{\IR^n} f(x)<\infty$.
Suppose that $u\in C((\IR^n\times\{0\}) \cup (\IR^n\times(0,T)) )$, $2\le p\le \infty,$ solves
$$\G_p u=0,\;\;\mbox{in $\IR^n\times(0,T)$},$$ 
and as $R\rightarrow \infty$, 
\ben
\sup_{0\le |x|\le R,\;0\le t\le T}u(x,t)\le \left\{ \begin{array}{llr}  \exp\left( o( R^{p/(p-1)}\right), &\mbox{for $2\le p<\infty$, and}\\
 \exp\left( o( R^{4/3}\right),& \mbox{for $p=\infty$}.\end{array}\right.
\een
It follows that $\inf_{\IR^n}f(x)\le u(x,t) \le\sup_{\IR^n} f(x),\;\forall(x,t)\in \IR^n\times(0,T)$.
\end{thm}
\NI In this context, we also provide an example of a sub-solution that 
supports the optimality of the growth rate in the theorem. See Remark \ref{4.6}.

The proofs of Theorems \ref{1.6}, \ref{1.61} and \ref{1.7} employ appropriate auxiliary functions and the comparison principle. 

We have divided our work as follows. Section 2 contains definitions, some previously proven results and some useful calculations. Proofs of Theorems \ref{1.1} and {1.2} are in Section 3. Theorem \ref{1.3} is proven in Section 4. Section 5 contains a discussion of the eigenvalue problem for $\D_p$ in the viscosity setting and has relevance for Theorem \ref{1.1}. Also see \cite{BH}.

We thank the referee for reading the work and for the many suggestions that have improved the presentation.
\vsp
\section{Preliminaries and some observations}

We start this section with the notion of a viscosity solution, see \cite{CIL}. This will be followed by recalling some previously proven results and presenting calculations for some useful auxiliary functions.

The set $usc(A)$ denotes the
set of all upper semi-continuous functions on a set $A$ and $lsc(A)$ the set of all lower semi-continuous functions on $A$. We say $u\in usc(\Om_T),\;u>0,$ is a sub-solution of $\G_p w=0$, in $\Om_T$, or $\G_p u\ge 0$ (see (\ref{1.2}))
if for any function $\psi(x,t)$, $C^2$ in $x$ and $C^1$ in $t$, such that $u-\psi$ has a local maximum
at some $(y,s)\in \Om_T$, we have
$$\Df \psi(y,s)-(p-1)u(y,s)^{p-2} \psi_t(y,s)\ge 0.$$
Similarly, $u\in lsc(\Om_T),\;u>0,$ is a super-solution of $\G_p w=0$ in $\Om_T$ or $\G_pu\le 0$ (see (\ref{1.2})) 
if for any function $\psi(x,t)$, $C^2$ in $x$ and $C^1$ in $t$, such that $u-\psi$ has a local minimum
at some $(y,s)\in \Om_T$, we have $\Df \psi(y,s)-(p-1)u(y,s)^{p-2} \psi_t(y,s)\le 0.$ A function $u\in C(\Om_T)$ is a solution of $\G_p w=0$, in $\Om_T$, or $\G_p u=0$, if $u$ is both a sub-solution and a super-solution. Analogous definitions can be provided for the equation $G_p w=0$, see (\ref{1.5}).

Next, we say $u\in usc(\Om_T\cup P_T),\;u>0,$ is a viscosity sub-solution of (\ref{1.4}) if $\G_p u\ge 0$, in $\Om_T,$ and $u\le h$ on $P_T$. Similarly, $u\in lsc(\Om_T\cup P_T),\;u>0,$ is a viscosity super-solution of (\ref{1.4}) if
$\G_p u\le 0$ in $\Om_T$, and $u\ge h$ on $P_T$. A function $u\in C(\Om_T\cup P_T),\;u>0,$ is a solution of (\ref{1.4}) if $\G_p u=0$ in $\Om_T$ and $u=h$ on $P_T$.

From hereon, all sub-solutions, super-solutions and solutions are to be taken in the viscosity sense.

We now recall some previously proven results.
See Section 3 in \cite{BM1, BM2} for proofs of Lemmas \ref{2.1}, \ref{2.3} and \ref{2.5}, and Theorems \ref{2.4} and \ref{2.2}. Lemma \ref{2.1} and Theorem \ref{2.4} hold regardless of sign of $u$.

\begin{lem}\label{2.1}{(Maximum principle)} Let $\Om\subset \IR^n,\;n\ge 2$ be a bounded domain and $T>0$.\\
(a) If $u\in usc(\Om_T\cup P_T)$ solves 
$$\D_p u- (p-1)|u|^{p-2} u_t\ge 0,\;\;2\le p<\infty,\;\;\mbox{or}\;\;\Dy u-3u^2 u_t\ge 0,\;\; \mbox{in $\Om_T$},$$
then $\sup_{\Om_{T}}u\le\sup_{P_T} u=\sup_{\Om_T\cup P_T}u.$\\
(b) If $u\in lsc(\Om_T\cup P_T)$ and 
$$\D_p u- (p-1)|u|^{p-2}u_t\le 0,\;\;2\le p<\infty,\;\;\mbox{or}\;\;\Dy u-3u^2u_t\le 0,\;\;\mbox{in $\Om_T$},$$ 
then 
$\inf_{\Om_T} u\ge \inf_{P_T}u=\inf_{\Om_T\cup P_T}u.$
\end{lem}
We present a comparison principle for $G_p$ (see (\ref{1.5})) that leads to Theorem \ref{2.2}. 

\begin{thm}\label{2.4}{(Comparison Principle)} 
Let $2\le p\le \infty$. Suppose that $\Om\subset \IR^n,\;n\ge 2,$ is a bounded domain and $T>0$. Let $u\in usc(\Om_T\cup P_T)$ and $v\in lsc(\Om_T\cup P_T)$ satisfy
$$G_p u\ge 0,\;\;\mbox{and}\;\;G_p v\le 0,\;\;\mbox{in $\Om_T$}.$$
If $u,\;v$ are bounded and $u\le v$ on $P_T$, then $u\le v$ in $\Om_T$. 
\end{thm}

The next is a comparison principle for $\G_p$ (see (\ref{1.2})) that applies to positive solutions. 

\begin{thm}\label{2.2}{(Comparison Principle}) 
Suppose that $\Om\subset \IR^n,\;n\ge 2,$ is a bounded domain and $T>0$. Let $u\in usc(\Om_T\cup P_T)$ and $v\in lsc(\Om_T\cup P_T)$ satisfy
$$\G_p u\ge 0,\;\;\mbox{and}\;\;\G_p v\le 0,\;\;\mbox{in $\Om_T$},\;\;\;\;2\le p\le \infty.$$
Assume that $\min(\inf_{\Om_T\cup P_T} u, \inf_{\Om_T\cup P_T} v)>0$. 
If $\sup_{P_T}v<\infty$ then
$$\sup_{\Om_T} u/v=\sup_{P_T}u/v.$$
In particular, if $u\le v$ on $P_T$, then $u\le v$ in $\Om_T$. Clearly, solutions to (\ref{1.4}) are unique.
\end{thm}

\begin{rem}\label{2.20}  We extend Theorem \ref{2.2} to the case $u\ge 0$ on $P_T$. Let $v$ be as in Theorem \ref{2.2}.

(i) If $u=0$ on $P_T$, then by Lemma \ref{2.1}, $u=0$, in $\Om_T$, and the conclusion holds.

(ii) Let $u\ge 0$ be a sub-solution (see Theorem \ref{2.2}) and $\sup_{\Om_T}u>0$; clearly, 
$\sup_{P_T}u>0$, by Lemma \ref{2.1}. 
Let $\ve>0$ be small. Define $u_\ve(x,t)=\max(u(x,t),\;\ve),\;\forall(x,t)\in \Om_t\cup P_T.$ We show that $u_\ve\in usc(\Om_T\cup P_T)$ and $\G_p u_\ve\ge 0,$ in $\Om_T$.

Let $(y,s)\in \Om_T\cup P_T$. Since $\limsup_{(x,t)\rightarrow (y,s)}u(x,t)\le u(y,s)$, we have
$\limsup_{(x,t)\rightarrow (y,s)}u_\ve(x,t)\le u_\ve(y,s)$ and $u_\ve\in usc(\Om_T\cup P_T)$. Next, let $\psi$, $C^2$ in $x$ and $C^1$ in $t$, and $(y,s)\in \Om_T$ be such that $u_\ve-\psi$ has a maximum at $(y,s)$. If $u_\ve(y,s)=u(y,s)(\ge \ve)$ then $u-\psi$ has a maximum at $(y,s)$ (since $u\le u_\ve$).
Since $u$ is sub-solution, 
we get 
$$\D_p\psi(y,s)-(p-1)u(y,s)^{p-2} \psi_t(y,s)=\D_p\psi(y,s)-(p-1)u_\ve(y,s)^{p-2} \psi_t(y,s)\ge 0.$$
Next, assume that $u_\ve(y,s)=\ve$. Rewriting $(u_\ve-\psi)(x,t)\le \ve-\psi(y,s)$,
$$0\le u_\ve(x,t)-\ve\le \langle D\psi(y,s), x-y\rangle+\psi_t(y,s)(t-s)+o(|x-y|+|t-s|),$$
as $(x,t)\rightarrow (y,s)$, where $(x,t)\in \Om_T$. Clearly, $D\psi(y,s)=0$ and $\psi_t(y,s)=0$. Thus,
$\D_p\psi(y,s)-(p-1)u_\ve(y,s)^{p-2} \psi_t(y,s)=0$, if $p>2$. A similar conclusion holds for $p=\infty$.
For $p=2$, we write the above Taylor expansion as $0\le \langle D^2\psi(y,s)(x-y),x-y\rangle/2 +o(|x-y|^2+|t-s|)$, as $(x,t)\rightarrow (y,s)$. It is clear that $\D\psi(y,s)\ge 0$. Thus, $u_\ve$ solves $\G_pu_\ve \ge 0$, in $\Om_T$, for $2\le p\le \infty$. 

We now apply Theorem \ref{2.2} to obtain that $\sup_{\Om_T}u_\ve/v\le \sup_{P_T}u_\ve/v$, for every small $\ve>0$. Since $u_\ve\le u+\ve$ (note that $u\ge 0$), we have
$\sup_{\Om_T}u/v\le \sup_{\Om_T}u_\ve/v\le \sup_{P_T}(u+\ve)/v\le \sup_{P_T} u/v+\sup_{P_T} \ve/v.$
The conclusion of Theorem \ref{2.2} holds by letting $\ve\rightarrow 0.$ $\Box$
\end{rem}

Next we state a change of variables result which relates $\G_p$ to $G_p$, see (\ref{1.2}) and (\ref{1.5}).

\begin{lem}\label{2.3} Let $\Om\subset \IR^n,\;n\ge 2,$ be a domain and $T>0$, and $2\le p\le \infty$. Suppose $u:\Om_T\rightarrow \IR^+$ and $v:\Om_T\rightarrow \IR$ such that
$u=e^v$. The following hold.\\
 (a) $u\in usc(\Om_t\cup P_T)$ and $\G_p u\ge 0$ if and only if $v\in usc(\Om_T\cup P_T)$ and $G_p v\ge 0$.\\
 (b) $u\in lsc(\Om_t\cup P_T)$ and $\G_p u\le 0$ if and only if $v\in lsc(\Om_T\cup P_T)$ and $G_p v\le 0$.
\end{lem}
We now present a separation of variable result that will be used for proving Theorem \ref{1.6}.
See Lemma 2.14 in \cite{BM1} and Lemma 2.3 in \cite{BM2}.

\begin{lem}\label{2.5} Let $\lam\in \IR$, $\mu \in \IR$, $T>0$, and $\psi:\Om\rightarrow \IR^+$.  

(a) Suppose that for some $2\le p<\infty$, $\psi\in usc(lsc)(\Om)$ solves
$\D_p\psi+\lam \psi^{p-1}\ge (\le) 0$ in $\Om$. If $u(x,t)=\psi(x,t) e^{-\mu t/(p-1) }$ then 
$\G_p u\ge (\le) 0,\;\mbox{where $\mu\ge(\le) \lam.$}$

(b) Suppose that $\psi\in usc(lsc)(\Om)$ solves
$\Dy \psi+\lam \psi^{3}\ge (\le) 0$ in $\Om$. If $u(x,t)=\psi(x,t) e^{-\mu t/3 }$ then 
$\G_\infty u\ge (\le) 0,\;\mbox{where $\mu\ge(\le) \lam.$} $
\end{lem}

We include two results that will be used in Theorem \ref{1.7}. We recall the radial form for the $\Df$, that is, if $r=|x|$, then, for $2\le p\le \infty$, 
\eqRef{2.50}
\Df v(r)=|v^{\prime}(r)|^{p-2} \left( (p-1) v^{\prime\prime}(r)+\frac{n-1}{r} v^{\prime}(r) \right)\;\;\mbox{and}\;\;\Dy u= \left(u^{\prime}(r)\right)^2 u^{\prime\prime}(r). 
\ee

\begin{lem}\label{2.6} Let $R>0$; set
$r=|x|,\;\forall x\in \IR^n$ and $h(x)=1-(r/R)^2,\; \forall\; 0\le r\le R.$

(i) For $2\le p<\infty$, take
$$k=p+n-2,\;\;\;\al=\frac{2p+k-1}{2(p-1)},\;\;\;\tht^2=\frac{k}{k+1}\;\;\;\mbox{and}
\;\;\;\lam_p= \frac{k\tht^{p-2}}{R^p} \left( \frac{2\al}{1-\tht^2}\right)^{p-1}.$$
Call $\eta(x)=h(x)^\al$ and $\phi_p(x,t)=\eta(x) e^{-\lam_p t/(p-1)},\;\;\forall\; 0\le r\le R.$

(ii) For $p=\infty$, define
$$\tht=1/\sqrt{2},\;\;\;\mbox{and}\;\;\;\lam_\infty=2^8/R^4. $$
Set $\eta=h(x)^2$ and $\phi_\infty(x,t)=\eta(x)e^{-\lam_\infty t/3},\;\forall 0\le r\le R.$

Then, for $2\le p\le \infty$, 
$\G_p \phi_p\ge 0,$ in $B_R(o)\times(0,\infty)$, $\phi(0,0)=1$ and $\phi(x, t)=0$, on $|x|=R$ and $t\ge 0.$
\end{lem}
\begin{proof} Our goal is to show that for $2\le p\le \infty$, $\Df \eta+\lam_p \eta^{p-1}\ge 0$, in $0\le r\le R$. 

{\bf Part (i): $2\le p<\infty$.} Observe that $\al>1$. Differentiating $\eta$, by using (\ref{2.50}), and
setting $H=h^{(\al-1)(p-1)-1}$, 
\bea\label{2.61}
\Df \eta&+&\lam_p \eta^{p-1}=\Df h^\al+\lam_p h^{\al(p-1)}\nonumber\\
&=&\al^{p-1}\left( h^{(\al-1)(p-1)} \Df h+(\al-1)(p-1)h^{(\al-1)(p-1)-1}|Dh|^p\right)
+\lam_p h^{\al(p-1)} \nonumber\\
&=&h^{(\al-1)(p-1)-1} \left[ \lam_p h^p+\al^{p-1} \left\{ (\al-1)(p-1)|Dh|^p+h\Df h\right\}\right]\nonumber\\
&=&H\left[ \lam_p h^p+\al^{p-1} \left\{ (\al-1)(p-1) \left(\frac{2r}{R^{2}}\right)^p - \frac{2^{p-1}r^{p-2}( p+n-2)h}{R^{2(p-1)}}  \right\} \right] \nonumber \\
&=&H \left[\lam_p h^{p} + \al^{p-1} \left\{ (\al-1)(p-1) \left(\frac{2r}{R^{2}}\right)^p- \frac{2^{p-1}r^{p-2}kh}{R^{2(p-1)}}
 \right\} \right].
\eea
We now estimate the right hand side of (\ref{2.61}) in $0\le r\le \tht R$ and in $\tht R\le r\le R$ separately.

In $0\le r\le \tht R$ disregard the middle term in (\ref{2.61}) and take $r=\tht R$ to see
\ben
\Df \eta+\lam_p \eta^{p-1}&\ge& H \left( \lam_p h^{p}- \frac{(2\al)^{p-1}r^{p-2}kh}{R^{2(p-1)}}\right)=hH \left( \lam_p h^{p-1}-\frac{(2\al)^{p-1}r^{p-2}k}{R^{2(p-1)}}\right)\\
&\ge &hH \left( \lam_p\left(1-\tht^2\right)^{p-1} -\frac{(2\al)^{p-1}\tht^{p-2}k}{R^p}\right)= 0.
\een

In $\tht R\le r\le R$ disregard the $\lam_p h^p$ term in (\ref{2.61}), set $r=\tht R$ in the second term and $h=1$ in the third term to obtain
\ben
\Df \eta+\lam_p \eta^{p-1}&\ge& \al^{p-1}H  \left\{ (p-1)(\al-1) \frac{2^{p}r^p}{R^{2p}}-\frac{2^{p-1}r^{p-2}k}{R^{2(p-1)}} h
 \right\}\\
&\ge & \frac{(2\al)^{p-1}Hr^{p-2}}{R^{2(p-1)}}  \left\{ \frac{2(\al-1)(p-1)r^2}{R^{2}}-k \right\}\\
&=&\frac{(2\al)^{p-1}Hr^{p-2}}{R^{2(p-1)}} \left\{ 2(\al-1)(p-1)\tht^2-k \right\}=0,
\een
since $2(p-1)(\al-1)\tht^2=k$.

{\bf Part(ii): $p=\infty$.} The work is similar to Part (i).
\bea\label{2.62}
\Dy \eta+\lam_\infty \eta^3&=&\Dy h^2+\lam_\infty h^6=8h^3\Dy h+ 8h^2|Dh|^4+\lam_\infty h^6\nonumber\\
&=&h^2\left[ \lam_\infty h^4+8\left( |Dh|^4+h\Dy h\right) \right] 
=h^2\left[ \lam_\infty h^4+8\left\{ \left(\frac{2r}{R^2}\right)^4-\frac{8r^2}{R^6}h\right\} \right]
\eea
We estimate (\ref{2.62}) in $0\le r\le \tht R$,
\ben
\Dy \eta+\lam_\infty \eta^3\ge h^2\left( \lam_\infty h^4-\frac{64r^2}{R^6}h \right)=h^3\left( \lam_\infty h^3-\frac{64r^2}{R^6} \right)\ge h^3\left( \lam_\infty (1-\tht^2)^3-\frac{64\tht^2}{R^4} \right)=0.
\een
From (\ref{2.62}), if $\tht R\le r\le R$ then
\ben
\Dy \eta+\lam_\infty \eta^3\ge 8h^2\left\{ \left(\frac{2r}{R^2}\right)^4-\frac{8r^2}{R^6}h\right\}
\ge \frac{64h^2r^2}{R^6} \left(\frac{2r^2}{R^2}-1\right)\ge \frac{64h^2r^2}{R^6} 
\left(2\tht^2-1\right)=0.
\een
The claim holds by an application of Lemma \ref{2.5}.
\end{proof}

We record a calculation we use in the various auxiliary functions we employ in our work.

\begin{rem}\label{2.70} Let $f(t)\in C^1,$ in $t\ge 0,$ and $f(t)\ge 0$. Set $r=|x|$ and 
$$u(x,t)=\pm f(t) r^{p/(p-1)},\;\;2\le p<\infty,\;\;\mbox{and}\;\;u(x,t)=\pm f(t) r^{4/3},\;\;p=\infty.$$
Call $A=n(p/(p-1))^{p-1}$ and $B=(p/(p-1))^p.$
We show that in $r\ge 0$, 
$$G_pu=\left\{ \begin{array}{lcr} \pm \left\{Af^{p-1}-(p-1)r^{p/(p-1)}f^{\prime}\right\}+(p-1)Bf^pr^{p/(p-1)},& 2\le p<\infty,\\ \pm\left\{(4^3/3^4)f^3-3f^{\prime}r^{4/3}\right\}
+(4/3)^4f^4 r^{4/3},& p=\infty.
\end{array}\right. $$
We prove the above for the $+$ case. The $-$ case can be shown similarly. Using (\ref{2.50}) the above holds in $r>0$ and $u\in C^2$ for $p=2$. We check at $r=0$ and for $2<p\le \infty$.

Suppose that $\psi$, $C^1$ in $t$ and $C^2$ in $x$, is such that $u-\psi\le u(o,s)-\psi(o,s)$, for $(x,t)$ near $(o,s)$. Thus,
$u(x,t)\le \langle D\psi(o,s), x\rangle+\psi_t(o,s)(t-s)+o(|x|+|t-s|),$ as $(x,t)\rightarrow (o,s)$. Clearly, $\psi_t(o,s)=0$ and $D\psi(o,s)=0$. Using the expansion
$$u(x,t)=f(t)|x|^{p/(p-1)}\le \langle D^2\psi(o,s)x, x\rangle/2+\psi_t(o,s)(t-s)+o(|x|^2+|t-s|),$$
as $(x,t)\rightarrow (o,s)$, we see that $D^2\psi(o,s)$ does not exist. Hence, $u$ is a sub-solution. 

Next, let $\psi$, $C^1$ in $t$ and $C^2$ in $x$, be such that $u-\psi\ge u(o,s)-\psi(o,s)$, for $(x,t)$ near $(o,s)$. Thus,
$u(x,t)\ge \langle D\psi(o,s), x\rangle+\psi_t(o,s)(t-s)+o(|x|+|t-s|),$ as $(x,t)\rightarrow (o,s)$. Clearly, $Du(o,s)=0$, $\psi_t(o,s)=0$ and $G_p\psi(o,s)=0$. 
Hence, $u$ is a super-solution. A similar argument works for $G_\infty$. $\Box$
\end{rem}
 
The next is an auxiliary function which is employed in the proof of Theorem \ref{1.7}.

\begin{lem}\label{2.7} Let $T>0$ and $2\le p\le \infty$. Set $r=|x|$ and for any fixed $\al>0$, define $\forall x\in \IR^n$ and any $0\le t\le T$, 
\ben
&&\phi_p(x,t)=\exp\left( a\left[ (t+1)^{\al (p-1)+1}-1\right]+ b (t+1)^\al r^{p/(p-1)} \right),\;\;2\le p<\infty,\\
&&\phi_\infty(x,t)=\exp\left( a [ (t+1)^{3\al+1}-1]+ b (t+1)^{\al}r^{4/3} \right),\;\;\;\;p=\infty.
\een

(i) For $2\le p<\infty$ take $a$ and $b$ such that
$$a= \frac{n p^{p-1}b^{p-1}}{(p-1)^p\{1+ \al (p-1)\}}\;\;\mbox{and}\;\;0<b^{p-1} \left( \frac{p}{p-1}\right)^p(T+1)^{\al (p-1)+1}<\al.$$

(ii) For $p=\infty$ take $a$ and $b$ such that
$$a=\frac{4^3b^3}{3^5(3\al+1)}\;\;\mbox{and}\;\; 0<b^3\left(\frac{4^4}{3^5}\right)(T+1)^{3\al+1}<\al.$$

Then for $2\le p\le \infty$, $\phi_p(o,0)=1$ and $\G_p \phi_p\le 0$, in $\IR^n\times (0,T)$.
\end{lem}
\begin{proof} We set $v=\log \phi_p$ and use Lemma \ref{2.3} and Remark \ref{2.70} to show that
$G_pv\le 0$, in $\IR^n\times(0,T).$ 

(i) Let $2\le p<\infty$. Then
$$v=\log\phi_p= a\left[ (t+1)^{\al (p-1)+1}-1\right]+ b (t+1)^\al r^{p/(p-1)}.$$
A simple calculation shows that $\Df r^{p/(p-1)}= n[ p/(p-1)]^{p-1},\;0\le r<\infty,\;2<p<\infty$ and  $\D r^2=2n,\;0\le r<\infty$, see Remark \ref{2.70}.

Using the above, calculating in $0\le r<\infty$ and $0< t< T$, and using the definitions of $a$ and $b$ we see that
\ben
&&\D_pv+(p-1)|Dv|^p-(p-1)v_t\\
&&=n\left(\frac{p}{p-1}\right)^{p-1}b^{p-1}(t+1)^{\al(p-1)} +(p-1) \left(\frac{p}{p-1}\right)^p b^p (t+1)^{\al p} r^{p/(p-1)}\\
&&\qquad\qquad\qquad\qquad\qquad -(p-1) \left\{ a (\al(p-1)+1) (t+1)^{\al(p-1)}+ \al b(t+1)^{\al-1}  r^{p/(p-1)} \right\}\\
&&=\left[ n\left( \frac{p }{p-1}\right)^{p-1}b^{p-1} -a(p-1)(\al(p-1)+1)\right] (t+1)^{\al(p-1)}\\
&&\qquad\qquad\qquad\qquad\qquad+ (p-1) r^{p/(p-1)}\left[ \left(\frac{p}{p-1}\right)^pb^p(t+1)^{\al p}-\al b (t+1)^{\al-1} \right]\\
&&=b(p-1)r^{p/(p-1)}(t+1)^{\al-1} \left(\left(\frac{p}{p-1}\right)^p b^{p-1}(t+1)^{\al (p-1)+1}-\al\right)\le 0.
\een
Thus, $\phi$ is a super-solution in $0\le r<\infty$ and $0<t<T$. 

We now show part (ii). Set
$$v=\log\phi_\infty= a [ (t+1)^{3\al+1}-1]+ b (t+1)^{\al}r^{4/3}.$$
Noting that $\Dy r^{4/3}=4^3/3^4$, in $0\le r<\infty$, and calculating,
\ben
\Dy v+|Dv|^4-3v_t&=&\frac{4^3}{3^4}b^3(t+1)^{3\al} +b^4 \left(\frac{4}{3}\right)^4(t+1)^{4\al}r^{4/3}\\
&&\qquad\qquad\qquad\quad\qquad\qquad -3\left\{a(3\al+1) (t+1)^{3\al}+b \al r^{4/3}(t+1)^{\al-1} \right\}\\
&&\le (t+1)^{3\al} \left( \frac{4^3b^3}{3^4}-3a(3\al+1)  \right)\\
&&\qquad\qquad\qquad\qquad+b r^{4/3}(t+1)^{\al-1}\left( b^3\left( \frac{4}{3}\right)^4(T+1)^{3\al+1}-3\al\right)\le 0,
\een
where we have used the definitions of $a$ and $b$. Rest of the proof is similar to part (i).
\end{proof}

We now extend existence results in \cite{BM1, BM2} to cylindrical domains $\Om\times (0,\infty)$. Set $\Om_\infty=\Om\times(0,\infty)$ and $P_\infty$ its parabolic boundary.
 
\begin{lem}\label{2.8} Let $\Om\subset \IR^n$ be bounded and $h\in C(P_\infty)$ with 
$0<\inf_{P_\infty}h\le \sup_{P_\infty}h<\infty$. Suppose that, for any $T>0$, 
$\G_p v=0,\;\mbox{in $\Om_T$ and $v=h$ on $P_T$,}\;2\le p\le \infty,$
has a unique positive solution. Then the problem
\eqRef{2.80}
\G_p v=0,\;\;\mbox{in $\Om{ _\infty}$ and $v=h$ on $P_\infty$,} 
\ee
has a unique positive solution $u\in C(\Om_\infty\times P_\infty).$ Moreover, 
$\inf_{P_\infty} h\le u \le \sup_{P_\infty}h.$

In particular, existence holds for any $\Om$, if $p>n$, and for any $\Om$ satisfying a uniform outer ball condition, if $2\le p\le n.$
\end{lem}
\begin{proof}
For any $T>0$ call $u_T$ to be the unique positive solution of $\D_p u_T-(p-1) u_T^{p-2} (u_T)_t=0,\;\mbox{in $\Om_T$, $u_T=h$ on $P_T$.}$

By Theorem \ref{2.2}, $u_{T_1}=u_{T}$ in $\Om_{T_1}$, for any $0<T_1<T$. Define
$u=\lim_{T\rightarrow \infty} u_T.$
Hence, $u$ solves the problem in $\Om_\infty$. To show uniqueness, if $v$ is any other positive solution, then
$v=u_T=u$ in $\Om_T$, by using Theorem \ref{2.2}. The maximum principle in Lemma \ref{2.1} shows that $\inf_{P_\infty}h\le \inf_{P_T}h\le u_T\le \sup_{P_T}h\le \sup_{P_\infty} h.$
\end{proof}

\begin{rem}\label{2.9} We record the following kernel functions of $\G_p$, for $2\le p\le \infty$. Define the functions $K_p$, in $\IR^n\times(0,\infty)$, as follows.
\ben
&&K_p(x,t)=t^{-n/p(p-1)}\exp\left(- \left(\frac{p-1}{p^{p/(p-1)}}\right)\left(\frac{|x|^{p}}{t}\right)^{1/(p-1)} \right),\;\;\;2\le p<\infty,\\
&&K_\infty(x,t)=t^{-1/12} \exp\left( - \left(\frac{3}{4} \right)^{4/3}\left( \frac{r^4}{t} \right)^{1/3} \right),\;\;\;\;\;p=\infty.
\een
For $p=2$, $K_2(x,t)=t^{-n/2}\exp( -|x|^2/(4t))$ is the well-known heat kernel for the heat equation.
Also,
$$\lim_{t\rightarrow 0}K_p(x,t)=0,\;x\ne o,\;\;\lim_{t\rightarrow 0^+}K_p(0, t)=\infty,\;\;\mbox{and}\;\;\lim_{|x|+t\rightarrow \infty} K_p(x,t)=0.$$ 
We omit the proof that $\G_p K_p=0$. $\Box$
\end{rem}

\section{Proofs of Theorems \ref{1.6} and \ref{1.61}}

{\bf Proof of Theorem \ref{1.6}.}  The proof is a some what simplified version of the one in \cite{BM1} and we use Remark \ref{2.20}. Let $\lamo$ be the first eigenvalue of
$\Df$ on $\Om$, see Section 5.

Define $M=\sup_{P_\infty} h$. By Lemma \ref{2.8}, we obtain $u\ge 0$ and $\sup_{\Om_\infty}u\le M.$
For every $t>0$ set  
\eqRef{3.1}
\mu(t)=\sup_{\overline{\Om}} u(x,t)\;\;\;\mbox{and}\;\;\;\nu(t)=\sup_{\p\Om} g(x,t).
\ee

{\bf Part (i).}  We observe by Remarks \ref{6.13} and {\ref{6.15} that 
for a fixed $0<\lam<\lam_\Om$ one can find a solution $\psi_\lam\in C(\overline{\Om}),\;\psi_\lam>0,$ such that 
\eqRef{3.3}
\Df \psi_\lam+\lam \psi_\lam^{p-1}=0,\;\;\mbox{in $\Om$, and $\psi_\lam=M$, on $\p\Om$.}
\ee
By Remark \ref{6.10},
\eqRef{3.4}
\psi_\lam>M,\;\;\mbox{in $\Om$, and}\;\;\psi_\lam(x)\ge u(x,t),\;\forall x\in \overline{\Om}.
\ee 

We now construct an auxiliary function for the proof. Let $0<S<T<\infty$; define
\bea\label{3.40}
\beta(t,T)=\exp\left( \frac{\lam (T-t)}{p-1} \right),\;\;\forall\;S\le t\le T.
\eea
In the rest of Part (i), we always choose $S$ and $T$ such that $\beta(S,T)\ge 2.$ Next, define the function
\eqRef{3.5}
F(t; S, T)=\frac{1}{2}\left[1+ \frac{\beta(t,T)-1 }{ \beta(S,T)-1 }\right]=\frac{1}{2}\left[\frac{\beta(S,T)-2}{\beta(S,T)-1}+ \frac{\beta(t,T)}{ \beta(S,T)-1}\right],\;\;\forall t\in [S, T].
\ee
Using (\ref{3.40}) and (\ref{3.5}) we get $\forall t\in [S,T]$,
\eqRef{3.6}
F_t=-\frac{\lam \beta(t,T)}{2(p-1) (\beta(S,T)-1)},\;\;F(S; S,T)=1,\;\;F(T;S,T)=\frac{1}{2},\;\;\mbox{and}\;\;\frac{1}{2}\le F(t; S,T)\le 1.
\ee
Let $\phi=\psi_\lam(x) F(t; S, T),$ $\forall(x,t)\in \Om\times(S,T)$, where $\psi_\lam$ is as in (\ref{3.3}). Using Lemma \ref{2.5}, (\ref{3.5}) and (\ref{3.6}), we get
\bea\label{3.7}
\G_p\phi&=&F^{p-1}\Df \psi_\lam-(p-1)\psi_\lam^{p-1}F^{p-2} F_t=-\lam F^{p-2} \psi_\lam^{p-1}
\left(F{+}\frac{p-1}{\lam}F_t\right)\nonumber\\
&=&-\lam F^{p-2} \psi_\lam^{p-1}\left\{ F-  \frac{1}{2} 
\left(  \frac{\beta(t,T) }{ \beta(S,T)-1  } \right) \right\}\nonumber\\
&=&-\frac{\lam \psi_\lam^{p-1} F(t)^{p-2}}{2}  \left( \frac{\beta(S,T)-2 }{ \beta(S,T)-1  } \right)\le 0, \;\;\;\;\forall\;S<t< T,
\eea
where we have used that $\beta(S,T)\ge 2.$

Recalling (\ref{3.1}) and the hypothesis of the theorem, there are $1<T_1<\cdots<T_m<\cdots<\infty$ such that for $m=1,2,\cdots,$ 
\bea\label{3.8}
(a)\;\;\beta(T_m, T_{m+1})\ge 2,\;\;\mbox{and}\;\;(b)\;\;0\le \nu(t)\le \frac{M}{2^{m}},\;\forall t\ge T_m,
\eea
see (\ref{3.40}). Note that (\ref{3.8}) (a) implies that $\lim_{m\rightarrow \infty}T_m =\infty$.
For each $m=1,2,\cdots,$ set 
\bea\label{3.9}
&&\mbox{$I_m=\Om\times(T_m, T_{m+1})$, $J_m$ the parabolic boundary of $I_m$, }  \\
&&\eta_m(t)= F(t; T_m, T_{m+1}) ,\;\forall t\in [T_m,\;T_{m+1}], \;\;\mbox{and}\;\;\phi_m(x,t)=\frac{\psi_\lam(x) \eta_m(t)}{2^{m-1}},\;\forall(x,t)\in \overline{I}_m,\nonumber
\eea
see (\ref{3.5}). 

Taking $m=1$, $\phi_1(x,t)=\psi_\lam \eta_1(t),$ using (\ref{3.3}), (\ref{3.4}), (\ref{3.6}) and (\ref{3.9}),  
\eqRef{3.10}
\phi_1(x, T_1)\ge M,\;\forall x\in \overline{\Om}, \;\;\mbox{and}\;\;\;\frac{M}{2}\le \phi_1(x,t)\le M,\;\;\forall(x,t)\in \p\Om\times [T_1,T_2].
\ee
Also, by (\ref{3.4}) and (\ref{3.8})(b),
\eqRef{3.101}
u(x,T_1)\le \phi_1(x,T_1),\;\forall x\in \overline{\Om},\;\;\mbox{and}\;\;u(x,t)\le \nu(t)\le \frac{M}{2},\;\forall(x,t)\in \p\Om\times[T_1, \infty).
\ee
Thus, $u\le \phi_1$, on $J_1$, and as $\G_p \phi_1\le 0$, in $I_1$, (see (\ref{3.7})) Theorem \ref{2.2} and Remark \ref{2.20} imply that $u\le \phi_1(x,t)$ in $I_1.$ 
We claim that $u\le \phi_1(x,t)$ in $\overline{I}_1$ ($u$ is upper semi-continuous). Take $\hat{T}_2>T_2$ and near $T_2$. The function
$\hat{\phi}_1(x,t)=\psi_\lam(x) F(t; T_1, \hat{T}_2)$ (see (\ref{3.6}) and (\ref{3.9})) satisfies the conclusions in (\ref{3.10}) and (\ref{3.101}) if we replace $\phi_1$ by $\hat{\phi}_1$.
Thus, $u\le \hat{\phi}_1$ in $\Om\times(T_1,\hat{T}_2)$ and the conclusion that $u\le \phi_1$, in $\overline{I}_1$, now follows by letting $\hat{T}_2\rightarrow T_2$. 
Clearly,
\eqRef{3.102}
u(x,T_2)\le \phi_1(x,T_2)=\psi_\lam(x)\eta_1(T_2)=\frac{\psi_\lam(x)}{2},\;\forall x\in \overline{\Om},
\ee
where we have used (\ref{3.6}). Moreover, since $F$ is deceasing in $t$ (see (\ref{3.6})), recalling (\ref{3.1}), we have
$$\mu(t)\le \sup_{\Om}\psi_\lam,\;\forall t\in [T_1,\;T_2],\;\;\mbox{and}\;\;\mu(T_2)\le \frac{\sup_{\Om} \psi_\lam}{2}.$$ 

We now use induction and suppose that for some $m=1,2\cdots$, 
\eqRef{3.11}
u(x,T_m)\le \frac{\psi_\lam(x)}{2^{m-1}},\;\;\forall x\in \overline{\Om},
\ee
(note that (\ref{3.11}) holds for $m=1,2$, see (\ref{3.4}) and (\ref{3.102})). We will prove that
\eqRef{3.12}
u(x,t)\le \frac{\psi_\lam(x)}{2^{m-1}},\;\forall(x,t)\in \overline{I}_m, \;\;\;\mbox{and}\;\;\;\mu(T_{m+1})\le \frac{\psi_\lam(x) } {2^m}.
\ee
thus proving part (i) of the theorem. 

By (\ref{3.7}) and (\ref{3.9}), 
$\G_p\phi_m\le 0, \;\mbox{in $I_m$}.$ By (\ref{3.6}), (\ref{3.9}) and (\ref{3.11}), 
\ben
&&u(x,T_m)\le \frac{\psi_\lam(x)\eta_m(T_m)}{2^{m-1}}=\phi_m(x,T_m),\;\forall x\in \overline{\Om},\;\;\mbox{and}\\
&&0\le g(x,t)\le \frac{M}{2^{m}}\le \phi_m(x,t),\;\forall(x,t)\in \p\Om\times[T_m, T_{m+1}).
\een
Thus, $\phi_m\ge u$ on $J_m$. Using Theorem \ref{2.2} and Remark \ref{2.20}, $u\le \phi_m$ in 
$\overline{I}_m$, and using (\ref{3.6}) 
$$u(x,t)\le \frac{\psi_\lam(x) \eta_m(t)}{2^{m-1}}\le \frac{\psi_\lam(x)}{2^{m-1}},\;\forall(x,t)\in \overline{I}_m,\;\;\mbox{and}\;\;u(x,T_{m+1})\le \frac{\psi_\lam(x)}{2^m},\;\forall x\in \overline{\Om}.$$
Thus, (\ref{3.12}) holds and part (i) is proven.

{\bf Part (ii).} Let $g(x,t)=0$, $\forall(x,t)\in \p\Om\times[T_0,\infty)$, for some $T_0>0$. 
We make some elementary observations. From (\ref{3.1}) one sees that 
$$M=\sup_{P_\infty} h=\max( \sup_{\overline{\Om}} f, \; \sup_{\p\Om\times[0,T_0]} g(x,t)).$$
Lemma \ref{2.1} implies that $0\le u(x,t)\le M,\;\forall (x,t)\in \Om_{T},$ for any $T>0$. 

We claim that
$\mu(t)$ is decreasing in $[T_0,\infty)$. Let $T_0\le T<\hat{T}<\infty$. Since $g=0$ on
$\p\Om\times[T,\hat{T})$, by Lemma \ref{2.1}, $\sup_{\Om\times(T,\hat{T})}u\le \mu(T)$. Since $u\in usc(\Om_\infty\cup P_\infty),\;u\ge 0$, it follows
that $\mu(t)\le \mu(T),\;T<t<\hat{T}$. Combining this with Part (i), we obtain
\eqRef{3.13}
\mbox{$\mu(t)$ is decreasing, in $t\ge T_0$, and $\lim_{t\rightarrow \infty} \left(\sup_{\overline{\Om}}u(x,t)\right)=0.$}
\ee

Next, let $T>T_0$ be large enough so that $\mu(T)>0$ and small (if $\mu(T)=0$ Part (ii) holds by Lemma \ref{2.1} and (\ref{3.13})).
By Remarks \ref{6.13} and \ref{6.15}, for any $0<\lam<\lam_\Om$, there is a $\psi_\lam \in C(\overline{\Om}),\;\psi_\lam>0,$ that solves
$$\Df \psi_\lam+\lam |\psi_\lam|^{p-2}\psi_\lam=0,\;\;\mbox{in $\Om$, with $\psi_\lam=\mu(T)$ on $\p\Om$.}
$$
By Remark \ref{6.10}, $\psi_\lam\ge \mu(T)$ in $\overline{\Om}$ and $\psi_\lam(x)\ge u(x,T),\;\forall x\in \overline{\Om}.$

Call $D_T=\Om\times(T,\infty)$ and $Q_T$ its parabolic boundary. We fix $\lam<\lamo$, close to $\lamo$, in what  follows. Define
$$L(x,t)=\psi_\lam(x) \exp\left(-\lam (t-T)/(p-1)\right),\;\;\mbox{in $\overline{D}_T$,}$$
and note that
$$L(x,T)\ge u(x,T),\;\forall x\in \overline{\Om},\;\; L(x,t)>0,\;\forall(x,t)\in \p\Om\times[T,\infty).$$ 
By Lemma \ref{2.5}, $\G_p L=0$, in $D_T$. Since $L\ge u$ on $Q_T$, Theorem \ref{2.2} and Remark \ref{2.20}
imply that $L\ge u$ in $\overline{D}_T$ and 
$$\lim_{t\rightarrow \infty} \frac{\log(\sup_\Om u(x,t)  )}{t}\le  \lim_{t\rightarrow \infty} \frac{\log(\sup_\Om L(x,t)  )}{t}=  -\frac{\lam}{p-1}.$$
Choosing $\lam$ arbitrarily close to $\lam_\Om$, we conclude
$$\lim_{t\rightarrow \infty} \frac{\log(\sup_\Om u(x,t)  ) }{t}\le -\frac{\lamo}{p-1}.\;\;\;\;\;\;\;\Box$$

\begin{rem}\label{3.15}  The decay rate in Part (i) of Theorem \ref{1.6} may depend on the decay rate along $\p\Om\times(0,\infty)$. Take $\ell(t):\IR^+\rightarrow \IR^+\cup\{0\}$, 
 $C^1$ in $t$ and decreasing to $0$, as $t\rightarrow \infty$. 
 
Part (ii) shows that if $u=0$ on $\p\Om\times(0,\infty)$ then
the slowest rate of decay is $e^{-\lamo t/(p-1)}$. 
Let $\Om=B_R(o)$; set $u(x,t)=\psi(x)\exp(-\lamo t/(p-1))$, where $\psi>0$ is a first eigenfunction of $\Df$ on $B_R(o)$, see Remark \ref{6.19}. By Lemma \ref{2.5}, $\G_pu=0$, in $B_R(o)\times (0,\infty)$. The decay rate in Theorem \ref{1.6} is attained.
\quad $\Box$
\end{rem}

Before presenting the proof of Theorem \ref{1.61} we make a remark. 

\begin{rem}\label{3.17}  Let $0\le S<T$ and $O=\Om\times(S,T)$.
We look at three possibilities. Let $u\in C(\Om_\infty),\;u>0$ solves
$\G_pu=0,\;u(x,0)=f(x),\;\forall x\in \overline{\Om},$ and $g(x,t)=1,\;\forall(x,t)\in \p\Om\times[0,\infty).$ Set $\mu(t)=\sup_{\overline{\Om}} u(x,t)$ and $m(t)=\inf_{\overline{\Om}} u(x,t)$. We apply Lemma \ref{2.1}.

{\bf (a) $\inf_\Om f=1$:} For every $t>0$, $m(t)=1$ and $1\le \mu(t)\le \sup_\Om f$. Then $u(x,t)\le \mu(S)$, in $O$, and $\mu(T)\le \mu(S)$. Hence, $\mu(t)$ is decreasing in $t$.

{\bf (b) $\sup_\Om f= 1$:} Clearly, $\mu(t)=1$ and $m(t)\le 1$, for every $t>0$. 
Clearly, $m(t)$ is increasing in $t$, since $u(x,t)\ge m(S)$, in $O$, and $m(T)\ge m(S)$. 

{\bf (c) $\inf_\Om f\le 1\le \sup_\Om f$:} Then $m(t)\le 1\le \mu(t),\;\forall t>0$. Arguing as in (a) and (b) we see that $m(t)$ is increasing and $\mu(t)$ is decreasing in $t$. \qquad $\Box$
\end{rem}

{\bf Proof of Theorem \ref{1.61}.} Let $u>0$ be a solution as stated in Theorem \ref{1.61}. We assume that
$0<\inf_\Om f< 1< \sup_\Om f$ and set
\eqRef{3.18} 
m=\inf_\Om f\;\;\mbox{and}\;\;M=\sup_\Om f.
\ee
Let $B_R(z)$ be the out-ball of $\Om$, where $z\in \IR^n$; define $r=|x-z|$. Part (i) addresses the case $2\le p<\infty$, and Part (ii) discusses $p=\infty$. Recall
Remark \ref{2.70}.

{\bf Part (i): $2\le p<\infty$.} Call
\eqRef{3.19}
A=A(p,n)=n\left(p/(p-1)\right)^{p-1},\;\; B=(p-1)\left(p/(p-1)\right)^p R^{p/(p-1)}.
\ee

{\bf Upper Bound.} Let $T_0>0$, to be determined later. Recalling (\ref{3.19}), take 
\eqRef{3.20}
\phi(x,t)=\exp\left[ a \left( \frac{R^{p/(p-1)}-r^{p/(p-1)}+b}{(1+t)^{\al}} \right) \right],\;\;\;0\le r\le R,
\ee
where 
\bea\label{3.201}
&& (i)\;\;0<\al\le \frac{1}{p-2},\;\;\mbox{if $2<p<\infty$},\;\;(ii)\;\;0<\al<\infty,\;\;\mbox{if $p=2$,}\nonumber\\
&& ab=(1+T_0)^{\al}\log M\;\;\;\mbox{and}\;\;\;a=\frac{A(p-1) (1+T_0)^{\al}}{pB}.
\eea
To make the calculations easier, we use Lemma \ref{2.3} and work with $v=\log \phi$. 
Recalling that $G_p v=\D_p v+(p-1) |Dv|^p-(p-1) v_t ,$ using (\ref{3.19}), the value of $ab$ (see (\ref{3.201})) and setting $C=\al(p-1)$, we get in $0\le r\le R,$ and $t>0$,
\bea\label{3.21}
G_p v&\le&-\frac{A a^{p-1}}{(1+t)^{\al(p-1)}}+\frac{B  a^p}{(1+t)^{\al p }} + \frac{\al(p-1)a(R^{p/(p-1)}-r^{p/(p-1)}+b)}{(1+t)^{\al+1}}  \nonumber\\
&\le &  C\left(\frac{aR^{p/(p-1)}+(1+T_0)^{\al}\log M}{(1+t)^{\al+1}}\right) +\frac{B a^p}{(1+t)^{\al p}}
-\frac{A a^{p-1}}{(1+t)^{\al(p-1)}}\nonumber\\
&=&\frac{1}{(1+t)^{\al(p-1)}}\left[ C\left(\frac{aR^{p/(p-1)}+(1+T_0)^{\al}\log M}{(1+t)^{1-\al(p-2)}} \right)+a^{p-1}\left( \frac{B  a}{(1+t)^{\al}}
-A \right) \right].
\eea
Using (\ref{3.201}) and calling $K=K(\al,p,n,R)>0$ (see below) we calculate in $t\ge T_0$,
\ben
a^{p-1}\left( \frac{B a}{(1+t)^{\al}}-A \right)&=&a^{p-1}A \left( \frac{p-1}{p} \left(\frac{1+T_0}{1+t}\right)^{\al}
-1\right)
\le a^{p-1} A\left( \frac{p-1}{p}-1\right)\\
&=&-\frac{a^{p-1}A}{p}=-K(1+T_0)^{\al(p-1)}.
\een
Using the above in (\ref{3.21}) together with the value of $a$ in (\ref{3.201}), we obtain in $t\ge T_0$, 
\bea\label{3.22}
G_pv&\le& \frac{1}{(1+t)^{\al(p-1)}}\left[C\left( \frac{aR^{p/(p-1)}+(1+T_0)^\al\log M  }{(1+t)^{1-\al(p-2)}}\right) -K(1+T_0)^{\al(p-1)} \right] \nonumber\\
&\le&\frac{1}{(1+t)^{\al(p-1)}}\left( \frac{\bar{K}(1+T_0)^{\al}}{(1+T_0)^{1-\al(p-2)}} -K(1+T_0)^{\al(p-1)} \right)\nonumber\\
&\le& \left(\frac{1+T_0}{1+t}\right)^{\al(p-1)}\left(\frac{ \bar{K}}{1+T_0}-K\right),
\eea
where $\bar{K}=\bar{K}(\al,p,n,R, M)>0.$ Choose $T_0$, large enough, so that $G_p v\le 0$ and 
$\G_p\phi\le 0$ in $\Om\times(T_0,\infty)$. Using (\ref{3.20}) and (\ref{3.201}), we see that
$$\inf_{\Om}\phi(x,T_0)\ge \exp\left( \frac{ab}{(1+T_0)^\al} \right)=M\;\;\;\mbox{and}\;\;\;\inf_{\p\Om}\phi(x,t)\ge 1,\;\;\forall t>0.$$
By Theorem \ref{2.2} (or Theorem \ref{2.4}) we see that $u(x,t)\le \phi(x,t),\;\forall(x,t)\in \Om\times (T_0,\infty)$, and 
\eqRef{3.23}
\limsup_{t\rightarrow \infty} u(x,t)\le \lim_{t\rightarrow \infty} \phi(x,t)=1,\;\;\mbox{for any $x\in \Om$.}
\ee

{\bf Lower Bound.} Set in $B_R(z)\times(0,\infty)$,
\eqRef{3.24}
\varphi(x,t)=\exp\left[ -(1+T_1)^\al \left( \frac{R^{p/(p-1)}-r^{p/(p-1)}-\log m}{(1+t)^{\al}} \right) \right],
\ee
where 
$$(i) \;0<\al\le \frac{1}{p-2},\;\;\mbox{if $2<p<\infty,$}\;\;\;(ii)\;0<\al<\infty,\;\;\mbox{if $p=2$,}$$
where $T_1>0$ is to be determined later. 
Set $w=\log \varphi$, we get in $0\le r\le R$ and $t\ge T_1,$
\ben
G_p w&\ge& A\left(\frac{1+T_1}{1+t} \right)^{\al(p-1)} -\frac{ \al(p-1)(1+T_1)^\al\left( R^{p/(p-1)}-\log m\right)}{(1+t)^{\al+1}}\\
&=&\left(\frac{1+T_1}{(1+t)}\right)^{\al(p-1)}\left(A-\frac{\al(p-1) \left( R^{p/(p-1)}-\log m\right)}{ 
(1+T_1)^{\al(p-2)} (1+t)^{1-\al(p-2)}}\right)\\
&\ge & \left(\frac{1+T_1}{1+t}\right)^{\al(p-1)}\left( A-\frac{K}{1+T_1}\right),
\een
where $K=K(\al, p, R, m)$.
Thus, there is a $T_1=T_1(a,\al, p, m, R)$ such that $\G_p \varphi\ge 0$ in $\Om\times(T_1,\infty)$. Next, we observe that
$$0<\varphi(x,T_1)\le m,\;\forall x\in \Om,\;\;\;\mbox{and}\;\;\varphi(x,t)\le 1,\;\forall(x,t)\in \p\Om\times[T_1,\infty).$$
Clearly, $\varphi\le u$ in $P_\infty$ and Theorem \ref{2.2} implies that $\varphi\le u$ in $\Om \times(T_1,\infty)$. Hence,
\eqRef{3.25}
\liminf_{t\rightarrow \infty}u(x,t)\ge \lim_{t\rightarrow \infty}\varphi(x,t)=1, \;\forall x\in \Om.
\ee
Thus, (\ref{3.23}) and (\ref{3.25}) imply the claim.

{\bf Part(ii): $p=\infty$.} The proof is similar to that in Part (i) and we provide the construction of a super-solution and a sub-solution. Set
$$A= 4^3/3^4\;\;\;\;\mbox{and}\;\;\;B=\left(4/3\right)^4 R^{4/3}.$$

{\bf Upper Bound:} Take 
$$\phi(x,t)=\exp\left[ a \left( \frac{R^{4/3}-r^{4/3}+b}{(1+t)^{\al}} \right) \right],\;\;\;0\le r\le R,\;\forall t>0,$$
where 
$$0<\al\le \frac{1}{2},\;\;\;ab=(1+T_0)^\al \log M,\;\;\mbox{and}\;\;\;a=\frac{3A(1+T_0)^\al}{4B}.$$
The quantity $T_0>0$, large, is to be chosen later. 

As done in Part(i), write $v=\log \phi$ and recall that $G_\infty v=\Dy v+|Dv|^4-3v_t$. Using the values of $a$, $ab$ and calculating in $0\le r\le R,\;t\ge T_0$,
\ben
G_\infty v&=&-\frac{a^3 A}{(1+t)^{3\al}}+  \left( \frac{4}{3}\right)^4  \frac{a^4 r^{4/3}}{(1+t)^{4\al}}
+\frac{3\al a(R^{4/3}-r^{4/3}+b)}{(1+t)^{\al+1}}\\
&\le&\frac{1}{(1+t)^{3\al}}\left(\frac{3\al a(R^{4/3}+b)}{(1+t)^{1-2\al}}+ \frac{a^4 B}{(1+t)^{\al}}-a^3 A\right)\\
&\le&\frac{1}{(1+t)^{3\al}}\left[ 3\al\left( \frac{aR^{4/3}+(1+T_0)^\al\log M}{(1+T_0)^{1-2\al}}\right)+ a^3\left(\frac{a B}{(1+T_0)^{\al}}- A\right)\right]\\
&=&\frac{1}{(1+t)^{3\al}}\left[ \frac{C(1+T_0)^\al}{(1+T_0)^{1-2\al}}-D(1+T_0)^{3\al}\right]=\left(\frac{1+T_0}{1+t}\right)^{3\al}\left[ \frac{C}{(1+T_0)}-D\right].
\een
Here the constants $C=C(\al, R)>0$ and $D=D(\al, R)>0$. We now choose $T_0$ large enough so that $\G_\infty \phi\le 0$ in $\Om\times[T_0,\infty)$. Rest of the proof is similar to that in Part (i).

{\bf Lower Bound.} Set
$$
\varphi(x,t)=\exp\left[ -a \left( \frac{R^{4/3}-r^{4/3}+b}{(1+t)^{\al}} \right) \right],\;\;\;0\le r\le R,\;\forall t>0,
$$
where $a>0$, $b>0$,
$$ab=-\log m,\;\;\;\;0<\al\le \frac{1}{2},$$
and $a$ is to be chosen later. Defining $w=\log \varphi$ and differentiating, in $t>0$,
\ben
G_\infty w&\ge& \frac{a^3 A}{(1+t)^{3\al}}-\frac{ 3\al a\left( R^{4/3}+b\right)}{(1+t)^{\al+1}}=\frac{1}{(1+t)^{3\al}}\left( a^3A- \frac{3\al\left(aR^{4/3}-\log m\right)}{(1+t)^{1-2\al}} \right)\ge 0,
\een
if $a>0$ is chosen large enough. Rest of the proof is similar to that in Part (i).\qquad\qquad $\Box$

\section{Proof of Theorem \ref{1.7} and optimality}

We make use of Lemmas \ref{2.6}, \ref{2.7} and Remark \ref{2.70} to prove the theorem. 
Let $T>0;$ we work in $\IR^n\times(0,T),\;n\ge 2$. For $2\le p\le \infty$, let $u>0$ solve
$\G_p u=0$, in $\IR^n\times(0,T)$, and $u(x,0)=f(x),\;\forall x\in \IR^n$. Set
\eqRef{4.1}
m=\inf_{\IR^n} f\;\;\;\mbox{and}\;\;\;M=\sup_{\IR^n} f.
\ee
We assume that $0<m\le M<\infty$.

{\bf Proof of Theorem \ref{1.7}.} Let $R>0$ be large.

{\bf (i) Lower bound.} Fix $y\in \IR^n$ and set $r=|x-y|,\;\forall x\in \IR^n$. Recall Lemma \ref{2.6} and in $0\le r\le R$, take
\ben
(i)\; \phi_p(x,t)=\left(1- \frac{r^2}{R^2}\right)^\al \exp\left(-\frac{\lam_p t}{p-1}\right),\;\;2\le p<\infty,\;\;(ii)\; \phi_\infty(x,t)=\left(1- \frac{r^2}{R^2}\right)^2 \exp\left(-\frac{\lam_\infty t}{3}\right),
\een
where $\al=(3/2)+n/(2(p-1))$. Also, from Lemma \ref{2.6}, one can write the values of
\eqRef{4.2}
\lam_p=\frac{K_1}{R^p},\;\;2\le p<\infty,\;\;\mbox{and}\;\;\lam_\infty=\frac{K_2}{R^4},
\ee
where $K_1=K_1(p,n)$ and $K_2$ is a universal constant. 

Call $\hat{\phi}_p(x,t)=m\phi_p(r,t)$ in $\overline{B}_R(y)\times(0,T)$. By Lemma \ref{2.6}, we see that the function $\hat{\phi}_p$ is a sub-solution in
$B_R(y)\times (0,T)$, $\hat{\phi}_p(x,0)\le m$, and $\phi_p(x,t)=0$, in $|x-y|=R.$ Using Theorem \ref{2.2} and Remark \ref{2.20}, $\hat{\phi}_p(x,t)\le u(x,t),\;\forall(x,t)\in B_R(y)\times(0,T)$. Writing $\mu_p=\lam_p/(p-1)$, for $2\le p<\infty$, and $\mu_\infty=\lam_\infty/3$, get
$$\hat{\phi}_p(y,t)=m\exp\left(-\mu_p t\right)\le u(y,t),\;\;0\le t<T.$$
Using (\ref{4.2}) and letting $R\rightarrow \infty$, we get $u(y,t)\ge m.$ This shows the lower bound in the theorem.

{\bf (ii) Upper bound.} We use Lemma \ref{2.7} and recall Remark \ref{2.70}. Recall the expressions for $\phi_p$ and take
$\al=1$ to obtain $\phi_p(x,t)=\phi_p(r,t)$ as 
\bea\label{4.30}
&&\phi_p(x,t)=\exp\left( a\left[ (t+1)^{p}-1\right]+ b (t+1) r^{p/(p-1)} \right),\;\;2\le p<\infty,\nonumber\\
&&\phi_\infty(x,t)=\exp\left( a [ (t+1)^{4}-1]+ b (t+1)r^{4/3} \right),\;\;\;\;p=\infty,
\eea
where $r=|x|$.
Also, recall there are constants $K_1=K_1(p,n)$ and $K_2=K_2(p)$, and absolute constants
$K_3$ and $K_4$  such that
\bea\label{4.3}
&&(i)\;\;a= K_1 b^{p-1},\;\;\mbox{and}\;\;0<b^{p-1}<\frac{K_2}{(1+T)^{p}},
\;\;\mbox{for $2\le p<\infty$ and}\nonumber\\
&&(ii)\;\;a=K_3 b^3,\;\;\mbox{and}\;\;0<b^3<\frac{K_4}{(1+T)^4},\;\;\mbox{for $p=\infty$}.
\eea
Then $\phi_p(0,0)=1$ and $\G_p \phi_p\le 0$ in $\IR^n\times(0,T).$

Let $b=3\ve$ where $\ve>0$ is small so that the conditions in (\ref{4.3}) are satisfied. We get from (\ref{4.3}),
\eqRef{4.4}
a=(3\ve)^{p-1}K_1,\;\;\mbox{for $2\le p<\infty$, and}\;\;a=27\ve^3K_3,\;\;\mbox{for $p=\infty.$} 
\ee
Set 
$$\beta=\frac{p}{p-1},\;\;\mbox{for $2\le p<\infty,$ and}\;\;\beta=\frac{4}{3},\;\;\mbox{for $p=\infty$.}$$
Fix $y\in H$; set $r=|x-y|$ and $R>0$, so large that 
\ben
\sup_{0\le t\le T} u(x,t)\le \exp( \ve r^\beta), \;\;\mbox{for $r\ge R.$}
\een

Call $C_{R,T}=B_R(y)\times(0,T)$. 
Define $v_p(x,t)=M\phi_p(x,t),\;\forall(x,t)\in C_{R,T}$. Then
in $0\le r\le R$, for large enough $R$,
$$v_p(y,0)=M,\;\;v_p(x,0)\ge M,\;\;\mbox{and}\;\;v_p(x,t)\ge \exp\left(2\ve R^\beta\right),\;\mbox{on $|x-y|=R$ and $0\le t\le T$.}$$
By Theorem \ref{2.2}, $u(x,t)\le v_p(x,t)$ in$ C_{R,T}$. Using (\ref{4.30}), (\ref{4.4}) and $v_p(y,t)=M\phi_p(0,t)$, we get
$$u(y,t)\le v_p(y,t)=M \exp((3\ve)^{p-1}K),\;2\le p<\infty, \;\mbox{and}\;u(y,t)\le v_\infty(y,t)=M \exp(27\ve^{3}K),$$
where $K=K(p,T).$ Clearly, the above estimate holds for any $\ve>0$ and $u(y,t)\le M$. The upper bound in the theorem holds and we obtain the statement of the theorem.\qquad$\Box$
\begin{rem}\label{4.5} It is clear that an analogous version of the Phragm\'en-Lindel\"of property also holds for the operator $G_p$. 
\quad$\Box$
\end{rem}

Next, we address the optimality of Theorem \ref{1.7}. The optimality in the case $p=2$ is discussed in \cite{CP}(page 246) and \cite{Tych}. An example due to Tychonoff shows that the 
growth condition in Theorem \ref{1.7} is optimal for the heat equation. We discuss below the case $2<p\le \infty$.
\vsp
\begin{rem}\label{4.6}{(Optimality)}

{\bf Case (i) $2<p<\infty$:} We now construct an example for Trudinger's equation in $\IR^n\times(0,T)$. Let $2<p<\infty$. Set $r=|x|$ and consider the function
\eqRef{4.7}
F(x,t)=\exp\left( A\frac{r^{p/(p-1)}}{t^\ve}-\frac{1}{t^a} \right),\;\;0\le r<\infty\;\mbox{and}\;0<t<T,
\ee
where we choose
\eqRef{4.8}
a>\frac{1}{p-2},\;\;\ve=\frac{a+1}{p-1}\;\;\mbox{and}\quad A= \frac{p-1}{p} \left( \frac{a(p-1)}{n}\right)^{1/(p-1)}.
\ee
It follows easily that $0<\ve <a$. Note that our construction works only for $p>2$. We set $F(x,0)=\lim_{t\downarrow 0}F(x,t)=0$, for any $x\in \IR^n$. 

Our goal is to show that $F$ is a sub-solution in $\IR^n\times (0,T)$. To simplify our calculations, we use Lemma \ref{2.3} and show that
$G_p H\ge 0$ where 
$$H=\log F= A\frac{r^{p/(p-1)}}{t^\ve}-\frac{1}{t^a}.$$

For completeness, we provide details here. Also, see Lemma \ref{2.70}. Differentiating,
\ben
H_r(r,t)=\left(\frac{At^{-\ve} p}{p-1}\right)r^{1/(p-1)} ,\;H_{rr}(r,t)=\left(\frac{At^{-\ve}p}{p-1}\right)\frac{r^{(2-p)/(p-1)}}{p-1},\;\mbox{and}\;H_t=\frac{-\ve A r^{p/(p-1)}}{t^{\ve+1}}+\frac{a}{t^{a+1}}.
\een
Thus,
\ben
&&\D_pH+(p-1)|DH|^p=|H_r|^{p-2}\left( (p-1) H_{rr}+\frac{n-1}{r}H_r\right)+(p-1)|H_r|^p\\
&&=\left( \frac{At^{-\ve}p}{p-1}\right)^{p-1}r^{(p-2)/(p-1)} \left(r^{(2-p)/(p-1)}+(n-1)r^{(2-p)/(p-1)}\right)+(p-1)\left(\frac{At^{-\ve}p}{p-1}\right)^p r^{p/(p-1)} \\
&&=\left( \frac{Ap}{p-1}\right)^{p-1}\frac{n}{t^{\ve(p-1)}}+(p-1)\left(\frac{Ap}{p-1}\right)^p \frac{r^{p/(p-1)}}{t^{\ve p}}.
\een
Using the above mentioned calculations and (\ref{4.8}), we obtain, in $\IR^n\times(0,T),$
\ben
G_p H=\D_pH+(p-1)|DH|^p-(p-1)H_t
\ge \left( \frac{Ap}{p-1}\right)^{p-1}\frac{n}{t^{\ve(p-1)}}-\frac{a(p-1)}{t^{a+1}}= 0.
\een

{\bf Case (ii) $p=\infty$:} Next we address the case $p=\infty$. This is similar to the Case (i). Let 
\eqRef{4.9}
a>1/2,\quad \ve=(a+1)/3\quad\mbox{and}\quad A=\left(3^5 a/4^3\right)^{1/3}.
\ee
Set $r=|x|$ and define
\be
F(x,t)=\exp\left( A\frac{r^{4/3}}{t^\ve}-\frac{1}{t^a} \right),\;\;\mbox{in $\IR^n\times(0,T)$}.
\ee
Set $F(x,0)= \lim_{t\downarrow 0} F(x,t)=0$, for any $x\in \IR^n$.  
As done before, we show that $G_\infty H\ge 0$ where 
$$H(x,t)=H(r,t)=\log F=A\frac{r^{4/3}}{t^\ve}-\frac{1}{t^a}.$$
Differentiating,
\ben
H_r=\left(\frac{4A t^{-\ve}}{3}\right)r^{1/3} ,\;\;\;H_{rr}(r,t)=\left(\frac{4At^{-\ve}}{9}\right)r^{-2/3}, \;\;\mbox{and}\;\;H_t=\frac{-A\ve r^{4/3}}{t^{\ve+1}}+\frac{a}{t^{a+1}}.
\een
Using the above and (\ref{4.9}), we obtain in $\IR^n\times(0,T)$,
\ben
&&G_\infty H=\D_\infty H+|DH|^4-3H_t=|H_r|^{2}H_{rr}+|H_r|^4-3H_t\\
&&=\left(\frac{4^3}{3^4}\right)\frac{A^3}{t^{3\ve}}+ \left( \frac{4}{3}\right)^{4}\frac{A^4r^{4/3}}{t^{4\ve}}+\frac{3A\ve r^{4/3}}{t^{\ve+1}}-\frac{3a}{t^{a+1}}  \ge   
\left(\frac{4^3}{3^4}\right)\frac{A^3}{t^{3\ve}}-\frac{3a}{t^{a+1}}=0.
\een

Finally, we get a sub-solution $\hat{F}=\max\{1, F\}$ which takes the value $1$ on $\IR^n\times\{0\}$. For a proof that $\hat{F}$ is a sub-solution, see part (ii) of Remark \ref{2.20}. $\Box$
\end{rem}
\vsp
\section{Positive solutions of $\Df u+\lam u^{p-1}=0,\;\;2\le p<\infty.$}

In the proofs of Parts (i) and (ii) of Theorem \ref{1.6}, we used the existence of a function $\psi_\lam$ for the problem in (\ref{3.3}). In order to make our work self-contained, we now address the question of existence of $\psi_\lam$ in the viscosity setting. We use ideas similar to those in \cite{BM0} which addresses the case of the infinity-Laplacian. We refer to \cite{CIL} for definitions.

The sets $usc(\Om)$ and $lsc(\Om)$ stand for the set of upper semi-continuous functions
and the set of lower semi-continuous functions in $\Om$, respectively. We assume $2\le p<\infty$ and $\Om\subset \IR^n$ is a bounded domain. To keep our work as brief as possible, we state our results as remarks.

\begin{rem}{(Maximum Principle)}\label{6.1}
Suppose that $u\in usc(lsc)(\overline{\Om})$ and $f:\Om\times\IR\rightarrow \IR$ is continuous. Assume that $f(x,t)<(>)0,\;\forall(x,t)\in \Om\times \IR$. 
$$\mbox{If $\Df u+f(x,u)\ge(\le)0$, in $\Om$, then $\sup_\Om u\le \sup_{\p\Om} u\;(\inf_\Om u\ge \inf_{\p\Om} u)$}.$$
\end{rem}
\begin{proof} We prove the maximum principle. Set $\ell=\sup_\Om u$ and $m=\sup_{\p\Om} u$, and assume that $\ell>m$. Let $\ve>0$ and $q\in \Om$ be such that $2\ve=\ell-m$ and $u(q)\ge\ell-\ve/2$. Call $\rho=
\sup_{x\in \p\Om}|x-q|$ and set
$$\psi(x)=\ell-\ve-\ve\left(|x-q|/\rho\right)^2,\;\;\forall x\in \overline{\Om}.$$ 
Thus, $(u-\psi)(q)>0$ and $(u-\psi)(x)\le m-(\ell-2\ve)=0,\;\forall x\in \p\Om.$ Noting that $u-\psi\in usc(\overline{\Om})$, let $z\in \Om$ be such that $u-\psi$ has a maximum. Using (\ref{2.50}),
$\Df \psi(z)+f(z,u(z))\le f(z, u(z))<0.$
This is a contradiction and the assertion holds. The proof of the minimum principle follows similarly.
\end{proof}

We prove a version of the strong maximum principle that is used in this work.

\begin{rem}{(Strong Maximum Principle)}\label{6.10} Let $f\in C(\Om\times\IR, \IR)$; assume that 
$\inf_\Om |f(x,t)|=0$ if and only if $t=0$. 

(a) Suppose that $f\le 0$ and $u\in usc(\overline{\Om})$ solves $\Df u+f(x,u)\ge 0$, in $\Om$. If 
$\sup_{\p\Om} u>0$ or $\sup_{\overline{\Om}}u<0$ then $u(x)<\sup_{\p\Om} u,\;\forall x\in \Om$.

(b) Suppose that $f\ge 0$ and $u\in lsc(\overline{\Om})$ solves $\Df u+f(x,u)\le 0$, in $\Om$. If 
$\inf_{\p\Om} u<0$ or $\inf_{\overline{\Om}}u>0$ then $u(x)>\inf_{\p\Om} u,\;\forall x\in \Om$.
\begin{proof} We show (a). Suppose that there is a point $z\in \Om$ such that 
$u(z)=\sup_{\overline{\Om}}u\ge \sup_{\p\Om}u.$ Clearly, $u(z)\ne 0.$ For $\ve>0$, small, define
$v_\ve(x)=u(z)+\ve |x-z|^2$, in $\Om$. Then $v_\ve \in C^2(\Om)$, $(u-v_\ve)(z)=0$ and
$$u-v_\ve =u(x)-u(z)-\ve |x-z|^2=u(x)-\sup_{\overline{\Om}}u-\ve|x-z|^2<0, \;\forall x\in \overline{\Om},\;x\ne z.$$
Thus, $z$ is the only point of maximum of $u-v_\ve$. Noting that $\D_p r^2=2^{p-1}r^{p-2}(p+n-2),\;r=|x-z|,$ and using the definition of a viscosity sub-solution, we get for $2<p<\infty$, 
$$(\Df v_\ve)(z)+f(z,u(z))=f(z,u(z))\ge 0\;\;\mbox{and}\;\;(\Delta v_\ve)(z)+f(z,u(z))=-2n\ve +f(z,u(z))\ge 0.$$ 
Letting $\ve\rightarrow 0$, 
$f(z,u(z))\ge 0$, and $u(z)=0$. This is a contradiction and the claim holds. Proof of (b) is similar.
 \end{proof}
\end{rem}

Let $S^{n\times n}$ be the set of symmetric $n\times n$ matrices and $Tr$ denote the trace of a matrix.
Define for $(q,X)\in \IR^n\times S^{n\times n}$ ($L_p$ is the differentiated version of $\Df$),
\eqRef{2.00}
L_p(q, X)=\left\{ \begin{array}{lcr} |q|^{p-2}Tr(X)+(p-2)|q|^{p-4} q_iq_jX_{ij},& p>2,\;q\ne 0,\\ Tr(X),& p=2,\\ 0& p>2, \;q=0. \end{array}\right.
\ee

\begin{rem}{(Comparison Principle)}\label{6.2}
Let $2\le p<\infty$, $f$ and $g\in C(\Om, \IR)$. Suppose that $u\in usc(\overline{\Om})$ and $v\in lsc(\overline{\Om})$ solve
$\D_p u+ f(x,u(x))\ge 0$ and $\D_p v+g(x, v(x))\le 0,$ in $\Om$.
If $sup_{\Om} (u-v)>\sup_{\p\Om} (u-v)$ then there is a point $z\in \Om$ such that
$$(u-v)(z)=\sup_\Om(u-v)\;\;\;\mbox{and}\;\;\;g(z, v(z))\le f(z, u(z)).$$
\end{rem}
\begin{proof} We adapt the proof in \cite{CIL} (also see \cite{BM0}) and provide a
brief outline. 

Set $M=\sup _{\Om}(u-v)$. Then one may find a point $z\in \overline{\Om}$ and sequences $x_{\ve}$ and $y_{\ve}$ such that (i) $M=(u-v)(z)$, and (ii) 
$x_\ve,\;y_\ve\rightarrow z$ as $\ve\rightarrow 0$. Moreover, since $M>\sup_{\partial \Om}(u-v)$, there is an open set $O$ such that $z,\; x_{\ve}$ and $y_{\ve}\in O\subset\subset \Om$. 
Also, there exist (see \cite{CIL})
  $X_{\ve},\;Y_{\ve}\in S^{n\times n}$ such that $((x_{\ve}-y_{\ve})/\ve, X_{\ve})\in \bar{J}^{2,+}u(x_{\ve})$ and $((x_{\ve}-y_{\ve})/\ve, Y_{\ve})\in \bar{J}^{2,-}v(y_{\ve})$. Moreover, $X_{\ve}\le Y_{\ve}$. Using the definitions of $\bar{J}^{2,+}$ and $\bar{J}^{2,-}$, we see that 
$$-f(x_{\ve},u(x_{\ve}))\le L_p((x_\ve-y_\ve)/\ve, X_\ve)  \le L_p((x_\ve-y_\ve)/\ve, Y_\ve) \le -g(y_{\ve}, v(y_{\ve})).$$
Now let $\ve\rightarrow 0$ to conclude that $g(z,v(z))\le f(z,u(z)).$  \end{proof}

Remark \ref{6.2} leads to the following comparison principle, see \cite{BM0}.

\begin{rem}{(Quotient Comparison)}\label{6.3} Let $u\in usc(\overline{\Om})$ and $v\in lsc(\overline{\Om})\cap L^{\infty}(\Om),\;v>0$. Suppose that $\lam$ and $\bar{\lam}$ are both positive. 

\NI (a) Let $\lam<\bar{\lam}$, $u$ and $v$ solve
$\Df u +\lam |u|^{p-2}u\ge 0\;\mbox{and}\;\Df v+\bar{\lam} v^{p-1}\le 0,\;\mbox{in $\Om$}.$
Then 
$$
\mbox{either $u\le 0$}\;\;\mbox{in $\Om$\qquad or }\;\;\;\;(u/v)(x)\le \sup_{\p\Om}(u/v)\;\;\mbox{in $\Om$}.$$
(b) Similarly, if $\bar{\lam}\le \lam$, $u>0$ and $v>0$ solve $\Df u-\lam u^{p-1}\ge 0\;\mbox{and}\;\Df v-\bar{\lam} v^{p-1}\le 0,\;\mbox{in $\Om$},$
then $\sup_{\Om} (u/v)\le \sup_{\p\Om}(u/v).$
\end{rem}
\begin{proof} Set $\mu=\sup_{\p\Om} (u/v)$ and $\nu=\sup_\Om (u/v)$. We observe that 
\eqRef{6.4}
u-\mu v\le 0\;\;\mbox{in $\p\Om$\qquad and}\;\;\;\;\;u-\nu v\le 0\;\;\mbox{in $\Om$}.
\ee
We prove (a). Assume that $\nu>0$ and $\mu<\nu$. Using (\ref{6.4}), 
$\sup_{\p\Om}(u-\nu v)<0$ and $\sup_\Om (u-\nu v)=0$.
Since $\Df (\nu v)+\bar{\lam}(\nu v)^{p-1}\le 0$, by Remark \ref{6.2}, we conclude that there is a point $y\in \Om$ such that $(u-\nu v)(y)=\sup_\Om(u-\nu v)=0$, implying that $u(y)>0$ and
$$\bar{\lam} (\nu v(y))^{p-1}\le \lam u(y)^{p-1}= \lam (\nu v(y))^{p-1}.$$
We have a contradiction and the assertion holds. 

To show (b), use (\ref{6.4}), $\mu<\nu$ to conclude that $\sup_\Om(u-\mu v)>0$. Since $\sup_{\p\Om}(u-\mu v)=0$ in $\p\Om$, Remark \ref{6.2} implies that there is a point $z\in \Om$ such that $(u-\mu v)(z)=\sup_\Om (u-\mu v)>0$ and
$$\lam u(z)^{p-1}\le \bar{\lam} (\mu v(z))^{p-1}<\bar{\lam} u(z)^{p-1}.$$
We have a contradiction and the assertion holds. \end{proof}

\begin{rem}\label{6.5} We extend the result in Remark \ref{6.3}(a) to include the case $\lam=\bar{\lam}$, that is, $\Df v+\lam v^{p-1}\le 0$ in $\Om$. 

Set $m=\inf_{\p\Om} v$, $M=\sup_\Om v$ and $v_t=v-t m$, where $0<t<1$. By Remark \ref{6.10}, $v> m$ and $v_t> (1-t)m>0$ in $\Om$. 
Since $\Df v\le -\lam v^{p-1}$, choose $\ve>0$, small (depending on $t$), so that
\bea\label{6.50}
\D_p v_t+(\lam+\ve) v_t^{p-1}&\le& v^{p-1} \left[-\lam+ (\lam+\ve) \left(\frac{v-t m}{v}\right)^{p-1} \right]\nonumber\\
&\le& v^{p-1} \left[ \frac{\lam+\ve}{\lam} \left(1-\frac{t m}{M}\right)^{p-1}-1 \right]\le 0.
\eea
Remark \ref{6.3}(a) holds for $\lam=\bar{\lam}$, since $u/v_t\le \sup_{\p\Om} u/v_t$ for any $0<t<1$. Moreover, (\ref{6.50}) leads to the estimates
\ben
(i)\;\; 0<\frac{\ve}{\lam}\le\left( \frac{1}{1-t(m/M)}\right)^{p-1}-1 \;\;\;\;\mbox{and}\;\;(ii)\;\; 1<\frac{M}{m}\le t \left( \frac{ (\lam+\ve)^{1/(p-1)}} { (\lam+\ve)^{1/(p-1)}-\lam^{1/(p-1)} }\right).\quad \Box
\een
\end{rem} 

\begin{rem}\label{6.6} Remark \ref{6.5} implies the following. Suppose that $\dl>0$ and $u\in C(\overline{\Om})$ solves
\eqRef{6.7}
\D_p u+\lam u^{p-1}=0\;\;\mbox{in $\Om$, $u>0$, and $u=\dl$ on $\p \Om$.}
\ee
For each $0<t<1$, there is an $\ve>0$, depending on $\sup_\Om u$, $\dl$ and $t$, such that $u_t=u-t\dl$ and 
$\D_p u_t+(\lam+\ve)u_t^{p-1}\le 0$ in $\Om$. Next, $\bar{u}_t=u_t/(1-t)$ is a super-solution of (\ref{6.7}) with $\lam$ replaced by $\lam+\ve$, and 
$\bar{u}_t=\dl$ on $\p \Om$. Also, $v(x)=\dl$ is a sub-solution. Both $v$ and $\bar{u}_t$ attain the boundary data in (\ref{6.7}) and $v\le \bar{u}_t$.
Remark \ref{6.5} and Perron's method (see \cite{CIL}) imply there is a $\psi\in C(\overline{\Om}),\;\psi>0,$ with
$\D_p \psi+(\lam+\ve)\psi^{p-1}=0$ in $\Om$, and $\psi=\dl$ on $\p \Om$. \quad $\Box$
\end{rem}

Let $\dl>0$. We now discuss existence of positive solutions $u\in C(\overline{\Om})$ to the problem
\eqRef{6.8}
\Df u+\lam u^{p-1}=0\;\mbox{in $\Om,$ and $u=\dl$ in $\p\Om$.}
\ee
We define
\eqRef{6.9}
E_\Om=\{\lam\ge 0:\;\mbox{problem (\ref{6.8}) has a positive solution $u$}\}\;\mbox{and}\;\;\lamo=\sup E_\Om.
\ee
We show in Remark \ref{6.11} below that $(0,\lamo)\subset E_\Om.$ Let $M_\lam=\sup_\Om u$, where $u$ solves (\ref{6.8}). Note that $u\ge \dl.$ We observe that if $0<\lamo<\infty$ (see Remark \ref{6.15}) and $0<\lam<\lamo$ then by Remark \ref{6.5}(i), for any $0<t<1$,
\eqRef{6.90}
0<\ve\le \lam \left[\left( \frac{1}{1-t(\dl/M_\lam)}\right)^{p-1}-1\right]\le \lamo -\lam\;\;\;\mbox{and}\;\;\;
M_\lam\ge \dl  \left( \frac{\lamo^{1/(p-1)}}{\lamo^{1/(p-1)}-\lam^{1/(p-1)} }\right).
\ee
Thus, $\lim_{\lam\rightarrow \lamo} M_\lam=\infty.$

Our goal is to show existence for small $\lam>0$ and to prove that $\lamo<\infty$. This would provide the information necessary for Theorem \ref{1.6}. Next, we show that 
(i) if $\lam\in E_\Om,\;\lam>0,$ then $[\lam,\lamo)\subset F_\Om$, and (ii) the domain monotonicity property of $\lamo$.

\begin{rem}\label{6.11} Let $E_\Om$ be as in (\ref{6.9}). Then $\lamo\not\in E_\Om$ and the following hold.\\
(i) If $\lam\in E_\Om$ then $(0,\lam^{\prime})\subset E_\Om$, for some $\lam^{\prime}>\lam.$ 
 Thus, $(0,\lamo)\subset E_\Om$.\\
(ii) If $O\subset\Om$ is a sub-domain then $\lamo\le \lam_O$.
\begin{proof} 
If $\lamo<\infty$ and $\lamo\in E_\Om$ then, by Remark \ref{6.7}, $\lamo$ will not be the supremum.

Part (i). Let $u_\lam\in C(\overline{\Om})$ solve
$\Df u_\lam+\lam u_\lam^{p-1}=0\;\mbox{in $\Om,\;u_\lam>0,$ and $u_\lam=\dl$ in $\p\Om.$}$
Clearly, $v=\dl$ in $\overline{\Om}$, is a sub-solution and $u_\lam$ is a super-solution of
\eqRef{6.12}
\Df w+\mu w^{p-1}=0\;\;\mbox{in $\Om,$ and $w=\dl$ in $\p\Om$,}
\ee
for any $0<\mu\le \lam$. Since both $v$ and $u_\lam$
attain the boundary data and $v\le u_\lam$,
recalling Remarks \ref{6.3}, \ref{6.5} and applying Perron's method, we obtain a positive solution of (\ref{6.12}) for each $0<\mu\le \lam.$ Combining this with Remark \ref{6.6} we see that $(0,\lam^{\prime})\subset E_\Om$ for some $\lam^{\prime}>\lam.$ Clearly, $(0,\lamo)\subset E_\Om.$

Part (ii). Assume that $\lam_O<\infty$ (otherwise we are done) and 
$\lamo>\lam_O$. By the definition of $E_\Om$ and Part (i), there is a $u\in C(\overline{\Om}),\;u>0,$ that solves 
$$\Df u+\lam_O u^{p-1}=0\;\;\mbox{in $\Om$, and $u=\dl$ in $\p\Om.$}$$
If $\lam<\lam_O$ then there is a unique positive solution $v_\lam$ to
$$\Df v_\lam +\lam v_\lam^{p-1}=0\;\;\mbox{in $O$, and $v_\lam=\dl$ in $\p O$.}$$
By Remark \ref{6.1}, $u\ge v_\lam$ on $\p O$, and by Remark \ref{6.3}, $u\ge v_\lam$ in $O$. Since this holds for any $\lam<\lam_O$, we apply (\ref{6.90})(on $O$) and let $\lam\rightarrow \lam_O$ to conclude that $u$ is unbounded. This is a contradiction and the claim holds. 
\end{proof} 
\end{rem}

We record a consequence of (\ref{6.90}) and Remark \ref{6.11}. 
\begin{rem}\label{6.13} Let $h\in C(\p\Om)$ with $\inf_{\p\Om} h>0$. Suppose $\lam>0$ is such that the problem
$$\Df u+\lam u^{p-1}=0\;\;\mbox{in $\Om$, $u=h$ in $\p\Om$,}$$
has a positive solution $u\in C(\overline{\Om})$. Call $E_{\Om,h}$ the set of all $\lam$'s for which the above has a positive solution. Set $\lam_{\Om,h}=\sup E_{\Om, h}$. We claim that
$$\lam_{\Om, h}\le \lamo,$$
where $\lamo$ is as in (\ref{6.9}). We comment that the two are equal and since the proof of equality requires existence we will not address it here, see Theorem \ref{1.6}.

\begin{proof} Assume that $\lamo<\infty$ and $\lamo<\lam_{\Om, h}$ (otherwise we are done). Thus, there  is a $\lam_1$ with 
$\lamo<\lam_1\le \lam_{\Om, h}$ and a function $u\in C(\overline{\Om}),\;u>0,$ so that 
$$\Df u+\lam_1 u^{p-1}=0\;\;\mbox{in $\Om$, and $u=h$ in $\p\Om$.}$$
By Remark \ref{6.11}(i), for any $0<\lam<\lamo$, there is a function $v_\lam$ so that $\Df v_\lam +\lam v_\lam^{p-1}=0$ in $\Om$, and $v_\lam=\dl$, in $\p\Om$,
where $0<\dl\le \inf_{\p\Om}h$. By Remark \ref{6.3}, $v_\lam\le u$ in $\Om$. Letting $\lam\rightarrow \lamo$ and applying (\ref{6.90}), we arrive at a contradiction. The claim holds.
\end{proof}
\end{rem}

We show existence for the problem in (\ref{6.8}). We assume that (i) for $2\le p\le n$, $\Om$ satisfies a uniform outer ball condition, and (ii) for $n<p<\infty$, $\Om$ is any domain.

\begin{rem}{(Existence:)}\label{6.13} Consider the problem of finding a positive solution
$u\in C(\overline{\Om})$ to 
\eqRef{6.14} \Df u+\lam u^{p-1}=0\;\;\mbox{in $\Om$, and $u=\dl$ on $\p\Om$.}
\ee
(i) Let $n<p<\infty$ and $\Om$ be any bounded domain. Then there is $\lam_0=\lam_0(p,n,\Om)>0$ 
such that (\ref{6.14}) has a solution $u$ for any $0<\lam<\lam_0$.\\
(ii) The same holds for $2\le p\le n$, if $\Om$ satisfies a uniform outer ball condition.

\begin{proof} The function $v=\dl$ is a sub-solution of (\ref{6.14}) for any 
$\lam>0$ and any $2\le p<\infty$. We construct super-solutions to (\ref{6.14}). Define
$R=\sup_{x,\;y\in \p\Om}|x-y|=\mbox{diam}(\Om).$

{\bf (i) $n<p<\infty$:} Fix $y\in \p\Om$, $0<\tht<1$ and set $r=|x-y|$. For $c>0$ (to be determined)
$$\al=\tht(p-n)/(p-1)\;\;\mbox{and}\;\;w_y(x)=\dl+c r^\al\;\;\forall x\in \overline{\Om}.$$
Using (\ref{2.50}), calculating in $0<r\le R$, 
\ben
\Df w_y=(c \al)^{p-1}r^{(\al-1)(p-2)+\al-2} \left( (p-1)(\al-1)+n-1 \right)=
-\frac{(c\al)^{p-1}(1-\tht)(p-n)}{r^{p-\al(p-1)}}\nonumber
\een
Using the above, we obtain in $\Om$,
\ben
\Df w_y+\lam w_y^{p-1}&\le& -\frac{(c\al)^{p-1}(1-\tht)(p-n)R^{\al(p-1)}}{R^{p}}+\lam (\dl+c R^\al)^{p-1}\\
&=&(\dl+cR^\al)^{p-1}\left[\lam-\left(\frac{cR^\al}{\dl+cR^\al}\right)^{p-1}\left( \frac{(1-\tht)(p-n)\al^{p-1}}{R^{p}}\right) \right].
\een
It is clear that if $0<\lam<(1-\tht)(p-n)\al^{p-1}R^{-p}$ then 
one can find a value of $c>0$ such that $w_y$ is a super-solution. Since $w_y(y)=\dl$ 
and $w_y\ge \dl$ on $\overline{\Om}$, using Remarks \ref{6.3}, \ref{6.5}, and applying Perron's method, the problem (\ref{6.14}) has a positive solution for $\lam>0$, small, and $E_\Om$ is non-empty.

{\bf (ii) $2\le p\le n$:} Let $\rho>0$ be the optimal radius of the outer ball. Fix $y\in \p\Om$ and let $z\in \IR^n\setminus \Om$ such that $B_\rho(z)\subset \IR^n\setminus \Om$ and $y\in \overline{B}_\rho(z)\cap \p\Om$. Set $r=|x-z|$ and take, for $c>0$,
$$\al>\max\left\{0,\; (n-p)/(p-1)\right\}\;\;\;\mbox{and}\;\;\;w_y(x)=\dl+c\left(\rho^{-\al}-r^{-\al}\right),\;\;\rho\le r\le R+\rho.$$
Using (\ref{2.50}), 
\ben
\Df w_y&=&(c\al)^{p-1}r^{-(\al+1)(p-2)-(\al+2)}\left( n-1-(\al+1)(p-1)\right)\\
&=&\frac{(c\al)^{p-1} (n-p-\al(p-1))}{r^{\al(p-1)+p}}
=-\frac{c^{p-1}k}{r^{\al(p-1)+p}},
\een
where $k=k(n,p,\al)>0$. Setting $J=\rho^{-\al}-(R+\rho)^{-\al}$ and using the above,
\ben
\Df w_y+\lam w_y^{p-1}&\le& -\frac{c^{p-1}k}{(\rho+R)^{\al(p-1)+p}}+\lam\left( \dl+cJ\right)^{p-1}\\
&=&\left( \dl+cJ \right)^{p-1}
\left[\lam-\frac{k}{(\rho+R)^{\al(p-1)+p}} \left( \frac{c}{ \dl+ cJ}\right)^{p-1}\right].
\een
Since $c/(\dl+cJ)<1/J$, one can find a $c>0$ such that $w_y$ is super-solution if 
$$0<\lam<\frac{k}{(R+\rho)^p}  \left( \frac{\rho^{\al}}{(R+\rho)^\al-\rho^\al }\right)^{p-1}.$$
Rest of the proof is as in Part (i). 
\end{proof}
\end{rem}

\begin{rem}{(Boundedness of $\lamo$)}\label{6.15} Remark \ref{6.13} shows that $\lamo>0$. We claim that $\lamo<\infty$. By Remark \ref{6.11}, this will follow if we show that $\lam_{B}<\infty$
for any ball $B$ in $\Om$. 
\begin{proof} For ease of presentation, we take the origin $o\in \Om$ and a ball $B_R(o)\subset \Om$.
Set $r=|x|$ and $\lam_R=\lam_{B_R}(o).$

Suppose that $\lam_R=\infty$. By (\ref{6.9}) and Remark \ref{6.11},
$(0,\infty)\subset E_{B_R(o)}.$ Let $\lam>0$ and $\lam_m=m^p\lam,\;m=0,1,\cdots,$. 
For each $m$, call $\phi_m>0$ the solution of 
\eqRef{6.16}
\D_p\phi_m+\lam_m\phi_m^{p-1}=0\;\;\mbox{in $B_R(o)$, and $\phi_m=\dl_m$ in $\p B_R(o)$}.
\ee
Here $\dl_m>0$ is so chosen that $\phi_m(o)=1$. Since $\Df$ is rotation and reflection invariant, 
applying Remark \ref{6.5} to reflections about $n-1$ planes through $o$, it follows that $\phi_m$ is radial. Next, using Remark \ref{6.1} in concentric balls, it is clear that $\phi_m$ is decreasing in $r$.

By Remarks \ref{6.10} and \ref{6.3}, if $\ell<m$ then $1=\phi_\ell(o)/\phi_m(o)\le \dl_\ell/\dl_m\;\mbox{and}\;\;0<\dl_m\le \dl_{m-1}\le \cdots\le \dl_1<1.$

Next, we scale as follows. For $m=1,2,\cdots,$ set $\forall x\in B_R(o)$,
$\psi_m(y)=\phi_m(x)$ where $y=mx.$ Thus, $\D_p\phi_m=m^p\D_p \psi_m$ and recalling $\lam_m=m^p \lam$, we get
$$\D_p\psi_m+\lam \psi_m^p=0\;\;\mbox{in $B_{mR}(o)$, $\psi_m(o)=1$ and $\psi_m(mR)=\dl_m$.}$$
Applying Remark \ref{6.5} in $B_{\ell R}(o)$, $\ell=1,2,\cdots, m-1$,
$$1=\frac{\psi_\ell(o)}{\psi_m(o)}\le \frac{\psi_\ell(r)}{\psi_m(r)}\le \frac{\dl_\ell}{\psi_m(\ell R)}\;\;\mbox{and}\;\;1\le  \frac{\psi_m(r)}{\psi_\ell(r)}\le \frac{\psi_m(\ell R)}{\dl_\ell}.$$
It is clear that $\psi_\ell(r)=\psi_m(r),\;0\le r\le \ell R,$ and $\dl_\ell=\psi_m(\ell R).$ For each $\ell=1,2,\cdots, m-1$, $\psi_m$ extends $\psi_\ell$ to $B_{mR}(o)$, and in particular, extends $\psi_1$ (defined on $B_R(o)$) to all 
$B_{mR}(o)$. 

Thus, for any $r>0$ we define $\psi_1(r)=\psi_m(r),\;\mbox{for any $m$ such that $mR>r$},$
thus extending $\psi_1$ to $\IR^n$. Also, $\psi_1$ is decreasing and $\psi_1(mR)=\dl_m,\;\;\forall m=1,2,\cdots.$

We claim that $\lim_{m\rightarrow \infty}\dl_m=0$. For $1\le \ell\le m$ and $0<\al<1$, we calculate (see (\ref{6.16})),
\ben
\Df \phi_m^\al&+&\lam_\ell \phi_m^{\al(p-1)}=\al^{p-1}\mbox{div}\left(\phi_m^{(\al-1)(p-1)}|D\phi_m|^{p-2}\phi_m\right)+\lam_\ell \phi_m^{\al(p-1)}\\
&=&\al^{p-1}\left( \phi^{(\al-1)(p-1)}_m\Df \phi_m+ (\al-1)(p-1)\phi_m^{\al (p-1)-p}|D\phi_m|^p\right)+  \lam_\ell\phi^{\al(p-1)}_m\\
&\le & -\al^{p-1}\lam_m \phi_m^{\al(p-1)}+\lam_\ell \phi_m^{\al(p-1)}=\left(\lam_\ell- \al^{p-1}\lam_m\right)\phi^{\al(p-1)}_m=0,
\een
if $\al=(\lam_\ell/\lam_m)^{1/(p-1)}=(\ell/m)^{p/(p-1)}$. 
Thus, $\phi_m^\al$ is a super-solution and Remark \ref{6.5} shows that
$1=\phi_\ell(o)/\phi^\al_m(0)\le \dl_\ell/\dl^\al_m.$ Using the value of $\al$, we have
$(\dl_m)^{(1/m)^{p/(p-1)}}\le (\dl_\ell)^{(1/\ell)^{p/(p-1)}},\;\forall \ell=1,2,\cdots,m-1.$
Hence, 
\eqRef{6.17}
\dl_m\le \dl_1^{m^{p/(p-1)}}.
\ee
Since $\dl_1<1$ (see Remark \ref{6.10} and (\ref{6.15})), 
$\lim_{m\rightarrow \infty}\psi_1(mR)=\lim_{m\rightarrow \infty}\dl_m=0$ and $\lim_{r\rightarrow \infty}\psi_1(r)=0.$

We now obtain lower bounds for $\psi_1$. Note that $\psi_1(o)=1$, $\psi_1$ is decreasing and 
$\psi_1(\ell R)=\dl_\ell,\;\forall\ell=1,2,\cdots$ (see above). For any $m=2,3,\cdots$, define in $R\le r\le 2mR$,
\ben
&&f_m(r)=\dl_1-\left(\dl_1-\dl_{2m}\right) (r^{\beta}-R^\beta)/((2mR)^\beta-R^\beta),\;\;\mbox{where}\;\;\beta=\frac{p-n}{p-1},\;p\ne n,\\
&&f_m(r)=\dl_1- (\dl_1-\dl_{2m})\log (r/R)/\log (2m) ,\;\;p=n.
\een
Then $f_m(R)=\dl_1$, $f_m(2mR)=\dl_{2m}$, $f_m> 0$ and $\D_p f_m=0$, in $B_{2mR}(o)\setminus B_R(o)$. Thus, $\D_p f_m+\lam f_m^{p-1}\ge 0$ in $B_{2mR}(o)\setminus B_R(o)$. By Remark \ref{6.5},
$f_m\le \psi_1$ in $B_{2mR}(o)\setminus B_R(o)$. Taking $r=mR$, we get, for large $m$, 
\ben
&&(i)\;f_m(mR)\ge \dl_1(1-2^{-\beta}), \;\;p>n,\;\;\;(ii)\;f_m(mR)\ge \dl_1(1-2^\beta)/m^{-\beta}, \;\;2\le p<n,\\
&&(iii)\;f_m(mR)\ge \dl_1\log 2/\log (2m), \;\;p=n.
\een
Since $\dl_m=\psi_1(mR)\ge f_m(mR)$, the above and (\ref{6.17}) lead to a contradiction. Thus, the claim holds and $\lam_R<\infty$ and $0<\lamo<\infty.$ 
\end{proof}
\end{rem}

\begin{rem}{(Scaling property)}\label{6.18} Let $\lam_R=\lam_{B_R(o)}.$ We claim that $\lam_R R^p=k$, for any $R>0$, where $k=k(p,n)>0.$ 
Let $R_1>0$, $R_2>0$ and $0<\lam<\lam_{R_1}$.  Suppose that $\phi>0$ solves $\D_p \phi+\lam \phi^{p-1}=0,$ in $B_{R_1}(o)$, with
$\phi_1=\dl$ on $\p B_{R_1}(o)$. Set $\psi(y)=\phi(x)$ where $y=R_2 x/R_1$. Then 
$\D_p\psi+\lam(R_1/R_2)^p \psi^{p-1}=0$ in $B_{R_2}(o)$, and $\psi=\dl$ on $\p B_{R_2}(o)$.
Clearly, $\lam_{R_1}R_1^p\le \lam_{R_2}R_2^p$. Replacing $R_1$ by $R_2$ shows equality.
$\Box$
\end{rem}

\begin{rem}\label{6.19}{(Eigenfunction)} The problem
$\D_p u+\lam_R u^{p-1}=0\;\mbox{in $B_R(o)$ and $u=0$ on $\p B_R(o)$,}$
has a positive solution $u$, a first eigenfunction, that is radial and decreasing.
\end{rem}
\begin{proof} Fix $0<\lam<\lam_R$. By Remark \ref{6.18}, let $\bar{R}$ be such that $\bar{R}^p\lam=\lam_R R^p$. Then $\bar{R}>R$. 

For each $k=1,2\cdots$, let (i) $0<\lam<\lam_k<\lam_R$ be such that $\lam_k\downarrow \lam$, (ii) 
$R_k=(\lam/\lam_k)^{1/p}\bar{R}$, and (iii) a unique function $u_k>0$ and $\dl_k>0$ such that (see Remark \ref{6.11})
\eqRef{6.20}
\D_p u_k+\lam u_k^{p-1}=0\;\;\mbox{in $B_{R_k}(o)$, $u_k(o)=1$ and $u_k=\dl_k$ on $\p B_{R_k}(o)$.}
\ee
As seen in Remark \ref{6.15}, $u_k$ is radial and decreasing. Also, 
$R_k<\bar{R}$, for each $k$, and $R_k\uparrow \bar{R}.$ Let $1\le \ell<k$. Applying 
Remark \ref{6.5}, we obtain 
$1\le v_\ell(r)/v_k(r)\le \dl_\ell/v_k(R_\ell)$
and $1\le v_k(r)/v_\ell(r)\le v_k(R_\ell)/\dl_\ell$ in $\overline{B_{R_\ell}(o)}$. Thus, $v_k(R_\ell)=\dl_\ell$, $v_k(r)=v_\ell(r)$ and $v_k$ extends $v_\ell$ to $B_{R_k}(o)$. 

For any $x\in B_{\bar{R}}(o)$, define $v(x)=v(r)=\lim_{k\rightarrow\infty} v_k(r)$. It clear that $v(x)=v_k(x)$ for any $k$ such that $|x|<R_k$. Also, $v(R_k)=\dl_k$. Since every $v_k$ is  decreasing in $r$, $v(r)$ is decreasing in $r$ and solves
$$\D_p v+\lam v^{p-1}=0\;\;\mbox{in $B_{\bar{R}}(o)$, $v>0$ and $v(o)=1$.}$$
Define $v(\bar{R})=\lim_{r\rightarrow \bar{R}}v(r).$ Thus, $v\in C(\overline{B}_{\bar{R}}(o)).$ Since 
$\lam=\lam_{\bar{R}}$, we have $v(\bar{R})=0$, otherwise, by Remark \ref{6.6}, $\lam_{\bar{R}}>\lam.$
Using scaling, we get existence of a radial eigenfunction on $B_R(o)$.
\end{proof}

\vsp
\NI Department of Mathematics, Western Kentucky University, Bowling Green, Ky 42101, USA\\
\NI Department of Liberal Arts, Savannah College of Arts and Design, Savannah, GA 31405, USA

\end{document}